\documentclass[12pt,A4]{article}
\usepackage[french,english]{babel}
\usepackage[T1]{fontenc}
\usepackage[applemac]{inputenc}
\usepackage{amssymb,url,xspace,amsmath,amsthm,eucal,yhmath}
\usepackage{mathrsfs,eufrak,array,epsfig,psfrag,graphicx,dsfont,bbm,stmaryrd,xcolor}
\usepackage[all]{xy}
\usepackage[pdfpagemode=UseNone,bookmarksopen=false,colorlinks=true,urlcolor=blue,citecolor=blue,citebordercolor=blue,linkcolor=blue]{hyperref}\usepackage{graphicx,pgf,tikz}
\usepackage{verbatim}
\usepackage {cases}
\usepackage[font=sf, labelfont={sf,bf}, margin=1cm]{caption}

\newtheorem{thmintro}{Theorem}
\newtheorem{corintro}{Corollary}

\newtheorem{thm}{Theorem}[section]
\newtheorem{lem}[thm]{Lemma}
\newtheorem{prop}[thm]{Proposition}
\newtheorem{cor}[thm]{Corollary}

\theoremstyle{definition}

\newtheorem{defn}[thm]{Definition}

\newtheorem{rem}[thm]{Remark}

\renewcommand\Pr[1]{\mathbb{P}\left(#1\right)}

\newcommand\Prmu[1]{\mathbb{P}_{\mu}\left(#1\right)}
\newcommand\Prmuj[1]{\mathbb{P}_{\mu,j}\left(#1\right)}

\newcommand\Es[1]{\mathbb{E}\left[#1\right]}
\newcommand\Esmu[1]{\mathbb{E}_{\mu}\left[#1\right]}

\newcommand \fl[1] {\left\lfloor #1 \right\rfloor}

\def \L {\mathbb{L}}

\def \P {\mathbb{P}}

\def \Pmu {\mathbb{P}_\mu}

\def \Pmuj {\mathbb{P}_{\mu,j}}
\def \Prhoj {\mathbb{P}_{\rho,j}}

\def \oW {\overline{W}}

\def \N {\mathbb N}
\def \E {\mathbb E}
\def \D {\mathbb D}

\def \R {\mathbb R}

\def \D {\mathbb D}
\def \T {\mathbb T}
\def \Z {\mathbb Z}

\def \W {\mathcal{W}}

\def \Znz {\Z/n\Z}

\def \Vc {\mathcal{V}}
\def \V {\textbf{V}}

\def \bx {\textnormal{\textbf{x}}}
\def \bf {\textbf{f}}

\def \l {\lambda}

\def \H { \mathcal{H}}

\def \d {\displaystyle}

\def \e {\epsilon}

\def \GW {\textnormal{GW}}
\def \z {\zeta}

\def \t {\mathfrak{t}}

\def \cZ { \mathcal{Z}}

\def \zt {|\tau|}
\def \lt {\lambda(\tau)}

\def \oW {\overline{W}}
\def \oX {\overline{X}}

\def \Z {\mathbb {Z}}

\long\def\symbolfootnote[#1]#2{\begingroup%
\def\thefootnote{\fnsymbol{footnote}}\footnote[#1]{#2}\endgroup}

\newcommand{\keywords}[1]{ \noindent {\footnotesize
             {\small \em Keywords.} {\sc #1}}.}

\newcommand{\ams}[2]{  \noindent {\footnotesize
             {\small \em AMS {\rm 2000} subject classifications.
             {\rm Primary {\sc #1}; secondary {\sc #2}}}.} }

\begin{document}

\vskip1cm

\centerline{\LARGE Limit theorems for conditioned non-generic Galton--Watson trees}

\vspace{8mm}

\centerline{\large Igor Kortchemski
\symbolfootnote[1]{Université Paris-Sud, Orsay, France. \texttt{igor.kortchemski@normalesup.org}} }

\vspace{4mm}

\centerline{\large March 2013}

 \vspace{3mm}
\vspace{1cm}

\begin{abstract}
We study a particular type of subcritical Galton--Watson trees, which are called non-generic trees in the physics community.  In contrast with the critical or supercritical case, it is known that condensation appears in certain large conditioned non-generic trees, meaning that with high probability there exists a unique vertex with macroscopic degree comparable to the total size of the tree. Using recent results concerning subexponential distributions, we investigate this phenomenon by studying scaling limits of such trees and show that the situation is completely different from the critical case. In particular, the height of such trees grows logarithmically in their size. We also study fluctuations around the condensation vertex.

\bigskip

\keywords{Condensation, Subcritical Galton--Watson trees, Scaling limits, Subexponential distributions}

 \bigskip

\ams{60J80,60F17}{05C80,05C05}
\end{abstract}

\vfill

\begin{figure}[h!]
 \begin{minipage}[c]{9cm}
   \centering
      \includegraphics[scale=0.4]{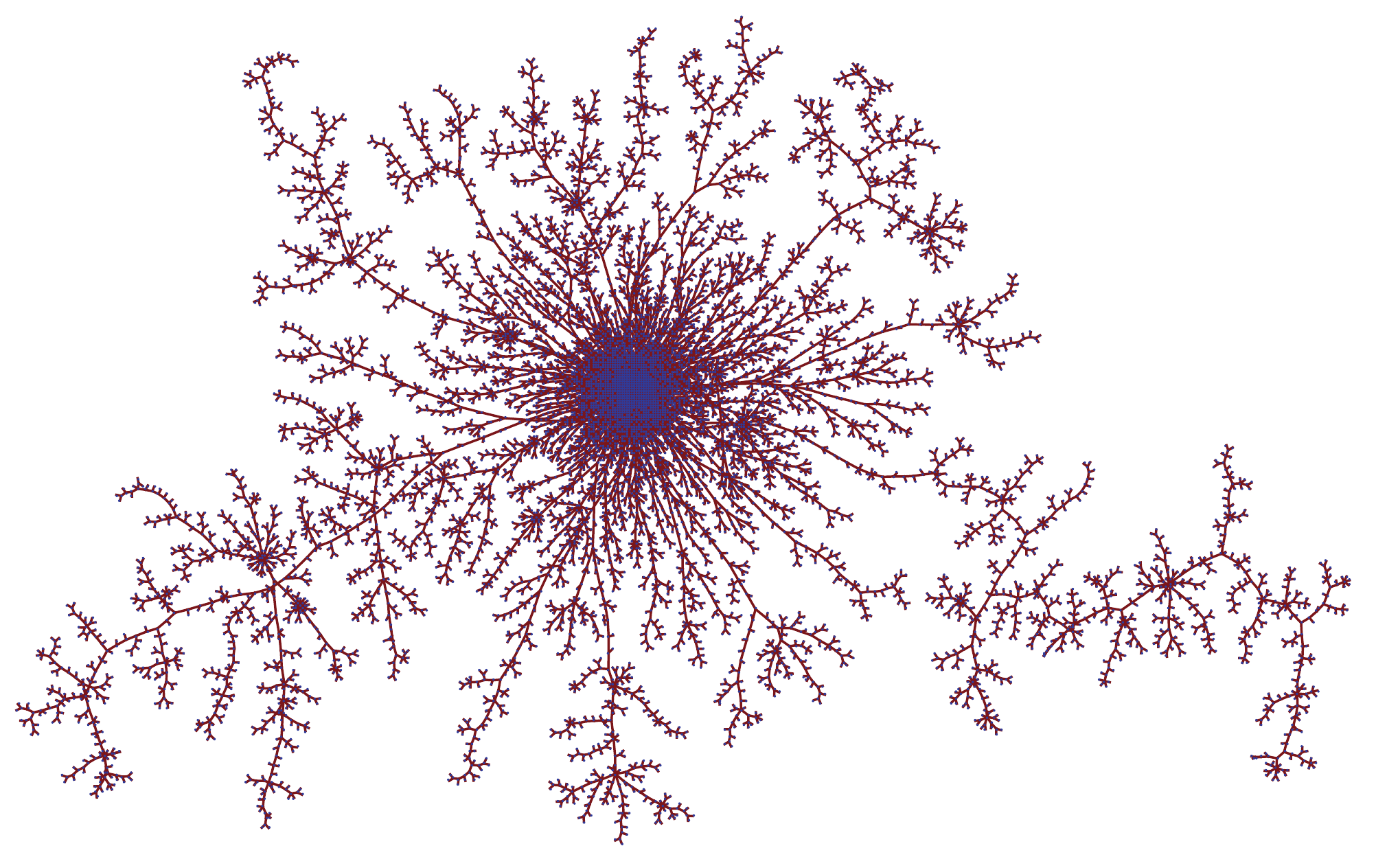}
   \end{minipage}
   \begin{minipage}[c]{9cm}
   \centering
      \includegraphics[scale=0.3]{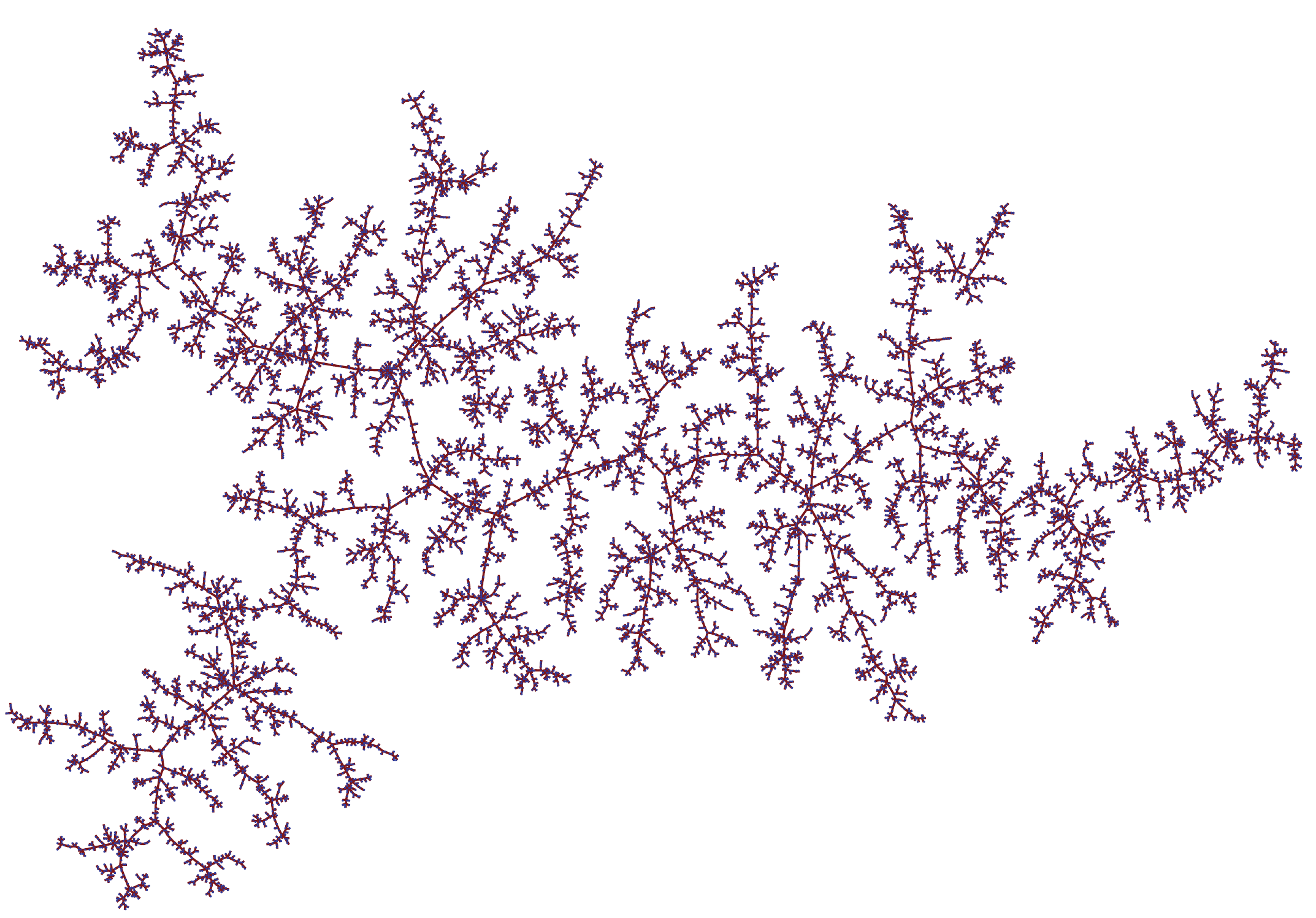}
   \end{minipage}
   \caption{\label {fig:condensation}The first figure shows a large non-generic Galton--Watson tree. The second figure shows a large critical Galton--Watson tree with finite variance.}
\end {figure}
\vfill

\pagebreak

\section*{Introduction}

The behavior of large Galton--Watson trees whose offspring distribution $ \mu= ( \mu_i) _ {i \geq 0}$ is \emph{critical} (meaning that the mean of $ \mu$ is $1$) and has \emph {finite variance} has drawn a lot of attention. If $ \t_n$ is a Galton--Watson tree with offspring distribution $ \mu$ (in short a $\GW_ \mu$ tree) conditioned on having total size $n$, Kesten \cite{Kes86} proved that $ \t_n$ converges locally in distribution as $n \to \infty$ to the so-called critical Galton--Watson tree conditioned to survive. Aldous \cite {Ald93} studied the scaled asymptotic behavior of $ \t_n$ by showing that the appropriately rescaled contour function of $ \t_n$ converges to the Brownian excursion.

These results have been extended in different directions. The ``finite second moment'' condition on $ \mu$ has been relaxed by Duquesne \cite {Du03}, who showed that when $ \mu$ belongs to the domain of attraction of a stable law of index $ \theta \in (1,2]$, the appropriately rescaled contour function of $ \t_n$ converges toward the normalized excursion of the $ \theta$-stable height process, which codes the so-called $  \theta$-stable tree (see also \cite{K11}). In a different direction, several authors have considered trees conditioned by other quantities than the total size, for example by the height \cite{KP96,LG10} or the number of leaves \cite {Riz11,Kor12}.

\paragraph {Non critical Galton--Watson trees.}  Kennedy \cite {Ken75} noticed that, under certain conditions, the study of non-critical offspring distributions reduces to the study of critical ones. More precisely, if $ \lambda>0$ is a fixed parameter such that $Z_ \lambda= \sum_ {i \geq 0}  \mu_i \lambda ^ i < \infty$, set ${\mu}^ {( \lambda)}_i= \mu_i \lambda ^ i /Z_ \lambda$ for $i \geq 0$. Then a $ \GW_ \mu$ tree conditioned on having total size $n$ has the same distribution as a $\GW_ {  {\mu}^ {( \lambda)}}$ tree conditioned on having total size $n$. Thus, if one can find $ \lambda>0$ such that both $Z_ \lambda< \infty$ and $  { \mu}^ {( \lambda)}$ is critical, then studying a conditioned non-critical Galton--Watson tree boils down to studying a critical one. This explains why the critical case has been extensively studied in the literature.

Let $\mu$ be a probability distribution such that $ \mu_0>0$ and $ \mu_k>0$ for some $k \geq 2$. We are interested in the case where there exist \emph{no} $ \lambda>0$ such that both $Z_ \lambda< \infty$ and $  { \mu}^ {( \lambda)}$ is critical (see \cite[Section 8]{Jan12} for a characterization of such probability distributions). An important example is when $ \mu$ is subcritical (i.e. of mean strictly less than $1$) and $ \mu_i \sim c/i ^ \beta$ as $i \to \infty$ for a fixed parameter $ \beta>2$. The study of such $\GW_ \mu$ trees conditioned on having a large fixed size was initiated only recently by Jonsson \& Stef\'ansson \cite{JS11} who called such trees \emph{non-generic} trees. They studied the above-mentioned case where  $ \mu_i \sim c/i ^ \beta$ as $i \to \infty$, with $ \beta>2$, and showed that if $ \t_n$ is a $ \GW_ \mu$ tree conditioned on having total size $n$, then with probability tending to $1$ as $ n \rightarrow \infty$, there exists a unique vertex of $ \t_n$ with maximal degree, which is asymptotic to $(1- \mathfrak {m})n$ where $\mathfrak {m}<1$ is the mean of $ \mu$.  This phenomenon is called  \emph {condensation} and appears in a variety of statistical mechanics models such as the Bose-Einstein condensation for bosons, the zero-range process \cite {JMP00,GSS03} or the Backgammon model \cite {BBJ97} (see Fig.~\ref{fig:condensation}).

Jonsson and Stef\'ansson \cite{JS11} have also constructed an infinite random tree $ \widehat{\mathcal {T}}$  (with a unique vertex of infinite degree) such that $ \t_n$ converges locally in distribution toward  $ \widehat{\mathcal {T}}$  (meaning roughly that the degree of every vertex of $ \t_n$ converges toward the degree of the corresponding vertex of $ \widehat{\mathcal {T}}$).  See Section \ref{section:location} below for the description of $ \widehat{\mathcal {T}}$. In \cite{Jan12}, Janson has extended this result to simply generated trees and has in particular given a very precise description of local properties of Galton--Watson trees conditioned on their size.

\bigskip

In this work, we are interested in the existence of scaling limits for the random trees $ \t_n$. When scaling limits exist, one often gets information concerning the global structure of the tree.

\paragraph {Notation and assumptions.} Throughout this work $\theta > 1$ will be a fixed parameter. We say that a probability distribution $(\mu_j)_{j \geq 0}$ on the nonnegative integers satisfies Assumption $(H_ \theta)$ if the following two conditions hold:
\begin{enumerate}
\item[(i)] $\mu$ is subcritical, meaning that $ \displaystyle 0<\sum_{j=0}^\infty j \mu_j < 1$.
\item[(ii)] There exists a measurable function $\mathcal{L}: \R_+ \rightarrow \R_+$ such that $ \mathcal {L}(x)>0$ for $x$ large enough and $\lim_{x
\rightarrow \infty} \mathcal{L}(tx)/\mathcal{L}(x)=1$ for all $t>0$ (such a function is called slowly varying) and $ \mu_n= \mathcal{L}(n)/ n ^ { 1+\theta}$ for every $ n \geq 1$.
\end{enumerate}
Let $ \zt$ be the total progeny or size of a tree $  \tau$. Condition (ii) implies that  $ \Prmu { \zt=n}>0$ for sufficiently large $n$. Note that (ii) is more general than the analogous assumption in \cite{JS11,Jan12}, where only the case $ \mathcal {L}(x) \rightarrow c$ as $x \rightarrow \infty$ was studied in detail. Throughout this text, $ \theta>1$ is a fixed parameter and $ \mu$ is a probability distribution on $ \Z_+$ satisfying the Assumption $ (H_ \theta)$. In addition, for every $n \geq 1$ such that $ \Prmu { \zt=n}>0$, $ \t_n$ is a $ \GW_ \mu$  tree conditioned on having $n$ vertices (note that $ \t_n$ is well defined for $n$ sufficiently large). The mean of $ \mu$ will be denoted by $\mathfrak {m}$ and we set $ \gamma = 1- \mathfrak {m}$.

\bigskip

We are now ready to state our main results which concern different aspects of non-generic trees. We are first interested in the condensation phenomenon and derive properties of the maximal degree. We then find the location of the vertex of maximal degree. Finally we investigate the global behavior of non-generic trees by studying their height.

\paragraph{Condensation.} If $ \tau$ is a (finite) tree, we denote  by $\Delta ( \tau)$ the maximal out-degree of a vertex of $ \tau$ (the out-degree of a vertex is by definition its number of children). If $\tau$ is a finite tree, let $u _ \star( \tau)$ be the smallest vertex (in the lexicographical order, see Definition \ref{def:fonctions} below) of $ \tau$ with maximal out-degree. The following result states that, with probability tending to $1$ as $n \rightarrow \infty$, there exists a vertex of $ \t_n$ with out-degree roughly $ \gamma n$  and that  the deviations around this typical value are of order roughly $O(n^ { 1/(2 \wedge \theta)})$, and also that the out-degrees of all the other vertices of $ \t_n$ are  of order roughly $O(n^ { 1/(2 \wedge \theta)})$. In particular the vertex with maximal out-degree is unique with probability tending to $1$ as $n \rightarrow \infty$.

\begin{thmintro}
\label{thmintro:1} There exists a slowly varying function $L$ such that if $B_n={L(n)n^ { 1/(2 \wedge \theta)}}$, the following assertions hold:
\begin{enumerate}
 \item[(i)] We have $ \d \frac{ \Delta( \t_n)}{ \gamma n}  \quad\mathop{\longrightarrow}^ {( \P)}_{n \rightarrow \infty} \quad 1$.
 \item[(ii)] Let $D_{n}$ be the maximal out-degree of vertices of $ \t_{n}$ except $u_{ \star}(\t_{n})$. If $ \theta \geq 2$, then $D_{n}/B_{n}$ converges in probability to $0$ as $ n \rightarrow \infty$. If $ \theta \in (1,2)$, then
 $$  \Pr {\frac{ D_{n}}{ B_n}\leq u}   \quad\mathop{\longrightarrow}_{n \rightarrow \infty} \quad  \exp \left( \frac{ 1}{\Gamma (1- \theta)} u^ {- \theta }  \right) \qquad u \geq 0,$$
where $ \Gamma$ is Euler's Gamma function.
 \item[(iii)] Let $(Y_t)_ {t \geq 0}$ be a spectrally positive Lévy process with Laplace exponent $ \E[ \exp(- \lambda Y_t)]= \exp(t \lambda ^ {2 \wedge \theta})$. Then $$ \frac{  \d \Delta( \t_n) - \gamma n}{B_n }  \quad\mathop{\longrightarrow} ^ {(d)}_{n \rightarrow \infty} \quad -Y_1.$$
  \end{enumerate}
\end{thmintro}

When $ \mu$ has finite variance $ \sigma^2 \in (0, \infty)$, one may take $B_n= \sigma \sqrt {n/2}$. Theorem \ref{thmintro:1} has already been proved when $ \mu_n \sim c/n^ { 1+ \theta}$ as $ n \rightarrow \infty$ (that is when $ \mathcal {L}=c+o(1)$, in which case one may choose $L$ to be a constant function) by Jonsson \& Stef\'ansson \cite{JS11} for (i) and Janson \cite{Jan12} for (iii). However, our techniques are different and are based on a coding of $ \t_n$ by a conditioned random walk combined with recent results of Armendáriz \& Loulakis \cite{AL11} concerning random walks whose jump distribution is subexponential, which imply that, informally, the tree $ \t_{n}$ looks like a finite spine of geometric length decorated with independent $ \GW_{ \mu}$ trees, and on top of which are grafted $ \Delta( \t_n)$ independent $ \GW_ \mu$ trees (see Proposition \ref{prop:tilde} and Corollary \ref{cor:indep} below for precise statements). The main advantage of this approach is that it enables us to obtain new results concerning the structure of $ \t_{n}$.

\paragraph{Localization of the vertex of maximal degree.} We are also interested in the location of the vertex of maximal degree $u_{ \star}( \t_{n})$.  Before stating our results, we need to introduce some notation. If $ \tau$ is a tree, let $U( \tau)$ be the index in the lexicographical order of the first vertex of $ \t_n$ with maximal out-degree (when the vertices of $ \tau$ are ordered starting from index $0$). Note that the number of children of  $u _ \star( \tau)$ is $ \Delta( \tau)$. Denote the generation of $u _ \star( \tau)$ by $|u _ \star( \tau)|$. 
\begin{thmintro} \label {thm:cvd}The following three convergences hold:
\begin{enumerate}
\item[(i)] For $i \geq 0$, $ \d \Pr {U( \t_n)=i}  \quad\mathop{\longrightarrow}_{n \rightarrow \infty} \quad \gamma \cdot \Prmu{ \zt \geq i+1}$.
 \item[(ii)] As $n \rightarrow \infty$, $|u _ \star( \t_n)|$ converges in distribution toward a geometric random variable of parameter $1- \mathfrak{m}$, i.e.
 $$\d\Pr {|u _ \star( \t_n)|=i}  \quad\mathop{\longrightarrow}_{n \rightarrow \infty} \quad (1-\mathfrak {m}) \mathfrak {m} ^ {i}, \qquad i \geq 0.$$
  \end{enumerate}
\end{thmintro}

\noindent See Section \ref{section:location} below for the description of $ \widehat{\mathcal {T}}$. Using  different methods, a  result similar to assertion (i) (as well as Propositions \ref{prop:invGW} and \ref{prop:luka} below) has been proved by Durrett \cite[Theorem 3.2]{D80} in the context of random walks when $ \t_n$ is a $ \GW_ \mu$ tree conditioned on having \emph{at least} $n$ vertices and in addition $\mu$ has finite variance. However, the so-called local conditioning by having a fixed number of vertices is often much more difficult to analyze (see e.g. \cite{Du03,LG10}). Note that $ \sum_ {i \geq 1} \Prmu { \zt \geq i}= \Esmu { \zt}= 1/ \gamma$, so that the limit in (i) is a probability distribution. The proof of (i) combines the coding of $ \t_n$ by a conditioned random walk with the previously mentioned results of Armendáriz \& Loulakis. The proof of the second assertion uses (i) together with the local convergence of $ \t_n$ toward the infinite random tree $\widehat{\mathcal {T}}$, which has been obtained by Jonsson \& Stef\'ansson \cite{JS11} in a particular case and then generalized by Janson \cite {Jan12}, and was already mentioned above.

\bigskip

We are also interested in the sizes of the subtrees grafted on $u_ \star ( \t_n)$. If $ \tau$ is a tree, for $1 \leq j \leq \Delta ( \tau)$, let $ \xi_{j}(\tau)$ be the number of descendants of the $j$-th child of $u_ \star ( \tau)$ and set $Z_j ( \tau)=\xi_{1}(\tau)+\xi_{2}(\tau) + \cdots + \xi_{j}(\tau)$. If $I$ is an interval, we let $\D(I , \R)$ denote the space of all right-continuous with left limits (càdlàg) functions $I \to \R$,
endowed with the Skorokhod $J_1$-topology (see \cite[chap. 3]{Bil99} and \cite[chap. VI]{JS03} for background concerning the Skorokhod topology).  If $ x \in \R$, let $\fl {x}$ denote the greatest integer smaller than or equal to $x$. Recall that $(Y_t)_ {t \geq 0}$ is the spectrally positive Lévy process with Laplace exponent $ \E[ \exp(- \lambda Y_t)]= \exp(t \lambda ^ {2 \wedge \theta})$. Recall the sequence $(B_n)_ {n \geq 1}$ from Theorem \ref {thmintro:1}.
 
 \begin {thmintro} \label {thm:fluctenfants} The following convergence holds in distribution in $ \D([ 0,1] , \R)$:
$$ \left( \frac{Z_ { \fl {  \Delta( \t_n) t}}( \t_n)-\Delta ( \t_n) t/ \gamma}{B_ n}, 0 \leq t \leq 1\right) \quad \mathop { \longrightarrow}^ {(d)}_ { n \rightarrow \infty} \quad
\d \left( \frac{1}{\gamma} Y_t, 0 \leq t \leq 1 \right).$$
\end {thmintro}
 \noindent Note that in the case when $ \mu$ has finite variance, we have $ \theta \geq 2$ and $Y$ is just a constant times standard Brownian motion. Let us mention that Theorem \ref{thm:fluctenfants} is used in \cite{JS12} to study scaling limits of random planar maps with a unique large face (see \cite[Proposition 3.1]{JS12}) and is also used in \cite{CKpercolooptrees} to study the shape of large supercritical site-percolation clusters on random triangulations.

\begin {corintro} \label {cor:intro}If $ \theta \geq 2$, $\max_ {1 \leq i \leq \Delta( \t_n)} \xi_{i}(\t_n)/B_n$ converges in probability toward $0$ as $n \rightarrow \infty$. If $ \theta<2$, for every $u > 0$ we have:
$$  \Pr {\frac{ \d1}{ B_n}\max_ {1 \leq i \leq \Delta( \t_n)} \xi_{i}(\t_n)\leq u}   \quad\mathop{\longrightarrow}_{n \rightarrow \infty} \quad  \exp \left( \frac{ 1}{ \gamma^ { \theta} \Gamma (1- \theta)} u^ {- \theta }  \right).$$
\end {corintro}

The dichotomy between the cases $\theta<2$ and $\theta \geq2$ arises from the fact that $Y$ is continuous if and only if $\theta\geq 2$.

\paragraph{Height of non-generic trees.} One of the main contributions of this work it to understand the growth of the height  $\H( \t_n)$ of $ \t_n$, which is by definition the maximal generation in $\t_n$. We establish the key fact that  $\H( \t_n)$ grows logarithmically in $n$:
  \begin {thmintro}\label {thmintro:height} For every sequence $( \lambda_n)_ { n \geq 1}$ of positive real numbers tending to infinity:
$$ \Pr { \left|\H( \t_n) -\frac{ \ln (n)}{ \ln(1/\mathfrak {m})} \right| \leq  \lambda_n}  \quad\mathop{\longrightarrow}_{n \rightarrow \infty} \quad1. $$\end {thmintro}
\noindent
Note that the situation is completely different from the critical case, where $ \H( \t_n)$ grows like a power of $n$. Theorem \ref {thmintro:height} implies that $ \H( \t_n)/  \ln(n) \to \ln(1/\mathfrak {m})$ in probability as $n \to \infty$, thus partially answering Problem 20.7 in \cite{Jan12}. Proposition \ref{prop:lp} below also shows that this convergence  holds in all the spaces $ \mathbb{L}_{p}$ for $p \geq 1$. Theorem \ref{thmintro:height} can be intuitively explained by the fact that the height of $ \t_{n}$ should be close to the maximum of the height of $  \gamma n$ independent subcritical $ \GW_ \mu$ trees, which is indeed of order $ \ln(n)$.

Since $ \t_{n}$ grows roughly as $ \ln(n)$ as $ n \rightarrow \infty$, it is natural to wonder if one could hope to obtain a scaling limit after rescaling the distances in $ \t_{n}$ by $ \ln(n)$. We show that the answer is negative and that we cannot hope to obtain a nontrivial scaling limit for $ \t_{n}$ for the Gromov--Hausdorff topology, in sharp contrast with the critical case (see \cite {Du03}). This partially answers a question of Janson \cite[Problem 20.11]{Jan12}.

\begin {thmintro} \label{thm:GH}
The sequence $ ( \ln(n)^{-1} \cdot \t_{n})_{n \geq 1}$ is not tight for the Gromov--Hausdorff topology, where $\ln(n)^{-1} \cdot \t_{n}$ stands for the metric space obtained from $ \t_{n}$ by multiplying all distances by $\ln(n)^{-1}$.
\end{thmintro}

The Gromov--Hausdorff topology is the topology on compact metric spaces (up to isometries) defined by the Gromov--Hausdorff distance, and is often used in the study of scaling limits of different classes of random graphs (see \cite[Chapter 7]{BBI01} for background concerning the Gromov--Hausdorff topology). 

However, we establish the convergence of the finite-dimensional marginal distributions of the height function coding $\t_{n}$. If $ \tau$ is a tree, for $0 \leq i \leq  \zt-1$, denote by $H_{i}( \tau)$ the generation of the $i$-th vertex of $ \tau$ in the lexicographical order.

\begin {thmintro} \label {thm:cvfinidim} Let $ k \geq 1$ be an integer and fix $0 < t_{1} < \cdots <  t_{k}<1$. Then
$$ (H_{ \fl{n t_{1}}}(\t_n), H_{\fl{n t_{2}}}(\t_n), \ldots, H_{\fl{n t_{k}}}( \t_{n}))  \quad\mathop{\longrightarrow}^{(d)}_{n \rightarrow \infty} \quad ( 1+\mathbf{e_{0}}+\mathbf{e_{1}}, 1+\mathbf{e_{0}}+\mathbf{e_{2}}, \ldots, 1+\mathbf{e_{0}}+\mathbf{e_{k}}),$$
where $(\mathbf{e_{i}})_{i \geq 0}$ is a sequence of i.i.d. geometric random variables of parameter $ 1 - \mathfrak{m}$ (that is $ \Pr {\mathbf{e _{0}}=i}=(1 - \mathfrak{m}) \mathfrak{m}^i$ for $i \geq 0$).
\end {thmintro}

Informally, the random variable $\mathbf{e}_{0}$ describes the length of the spine, and the random variables $(\mathbf{e}_{1}, \ldots, \mathbf{e}_{k})$ describe the height of vertices chosen in a forest of independent subcritical $ \GW_ \mu$ trees. Note that these finite-dimensional marginal distributions converge without scaling, even though the height of $ \t_{n}$ is of order $ \ln(n)$.

\bigskip

This text is organized as follows. We first recall the definition and basic properties of Galton--Watson trees. In Section 2, we establish limit theorems for large conditioned non-generic Galton--Watson trees. We conclude by giving  possible extensions and formulating some open problems.

\medskip

\paragraph {Acknowledgments.} I am grateful to Nicolas Broutin for interesting remarks, to Grégory Miermont for a careful reading of a preliminary version of this work, to Olivier Hénard for Remark \ref {rem:gen}, to Pierre Bertin and Nicolas Curien for useful discussions and to Jean-François Le Gall and an anonymous referee for many useful comments.

\bigskip

 \tableofcontents

\bigskip

\section{Galton--Watson trees}
\label {sec:GW}
\subsection {Basic definitions}

We briefly recall the
formalism of plane trees (also known in the literature as
rooted ordered trees) which can be found in \cite{LG06} for
example.

\begin{defn}Let $\Z_+=\{0,1,2,\ldots\}$ be the set of all nonnegative integers and let $\N$ be the set of all positive integers. Let also $U$ be the set of all labels defined by:
$$U = \bigcup_{n=0}^{\infty}(\N)^n$$
where by convention $(\N)^0 = \{\emptyset\}$. An element of $U$ is a
sequence $u=u_1\cdots u_k$ of positive integers and we set
$|u| = k$, which represents the \og generation \fg \, of $u$. If $u = u_1
\cdots u_i$ and $v = v_1 \cdots v_j$ belong to $U$, we write $uv= u _1
\cdots u_i v_1 \cdots v_j$ for the concatenation of $u$ and $v$. In
particular, we have $u \emptyset = \emptyset u = u$. Finally, a
\emph{plane tree} $\tau$ is a finite or infinite subset of $U$ such
that:
\begin{itemize}
\item[(i)] $\emptyset \in \tau$,
\item[(ii)] if $v \in \tau$ and $v=ui$ for some $i \in \N$, then $u
\in \tau$,
\item[(iii)]  for every $u \in \tau$, there exists $k_u(\tau) \in  \{0,1,2, \ldots \} \cup  \{ \infty\}$ (the number of children of $u$) such that, for every $j \in \N$, $uj \in \tau$ if and only
if $1 \leq j \leq k_u(\tau)$.\end{itemize}
Note that in contrast with \cite {LG05,LG06} we allow the possibility $k_u ( \tau)=\infty$ in (iii). In the following, by \emph{tree} we will always mean plane
tree, and we denote the set of all trees by $\T$ and the set of all finite trees by $\T_{f}$. We will often view each vertex of
a tree $\tau$ as an individual of a population whose $\tau$ is the
genealogical tree. The total progeny or size of $\tau$ will be
denoted by $|\tau|= \textrm{Card}(\tau)$.  If $\tau$ is a tree and
$u \in \tau$, we define the shift of $\tau$ at $u$ by $T_u \tau=\{v
\in U; \, uv \in \tau\}$, which is itself a tree.
\end{defn}

\begin{defn}Let $\rho$ be a probability measure on $\Z_+$. The law of the
Galton--Watson tree with offspring distribution $\rho$ is the unique
probability measure $\P_\rho$ on $\T$ such that:
\begin{itemize}
\item[(i)] $\P_\rho(k_\emptyset=j)=\rho(j)$ for $j \geq 0$,
\item[(ii)] for every $j \geq 1$ with $\rho(j)>0$, conditionally on $\{ k_{\varnothing}=j\}$, the subtrees
$T_1 \tau, \ldots, T_j \tau$ are independent and identically distributed with distribution $\P_\rho$.
\end{itemize}
A random tree whose distribution is $\P_\rho$ will be called a Galton--Watson tree with offspring distribution $ \rho$, or in short a $\GW_\rho$ tree.
\end{defn}

 In the sequel, for every integer $j \geq 1$, $\Prhoj$ will denote the probability measure on $\T^j$ which is the
distribution of $j$ independent $\GW_\rho$ trees. The canonical element
of $\T^j$ is denoted by $\bf$. For $\bf=(\tau_1,\ldots,\tau_j)
\in \T^j$, let $|\bf|=|\tau_1|+\cdots+|\tau_j|$ be the total progeny of $\bf$.
\subsection{Coding Galton--Watson trees}

We now explain how trees can be coded by three different functions.
These codings are important in the understanding of large
Galton--Watson trees.

\begin{figure*}[h!]
 \begin{minipage}[c]{9cm}
   \centering
      \includegraphics[scale=0.4]{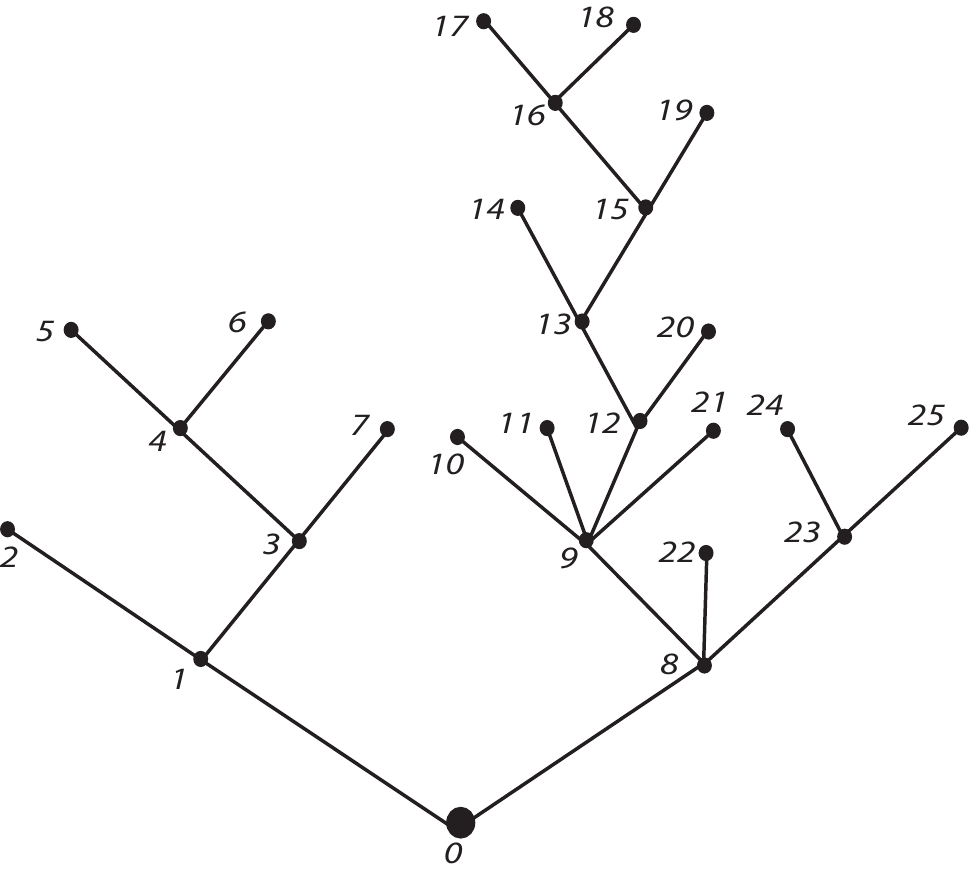}
   \end{minipage}
   \begin{minipage}[c]{9cm}
   \centering
      \includegraphics[scale=0.4]{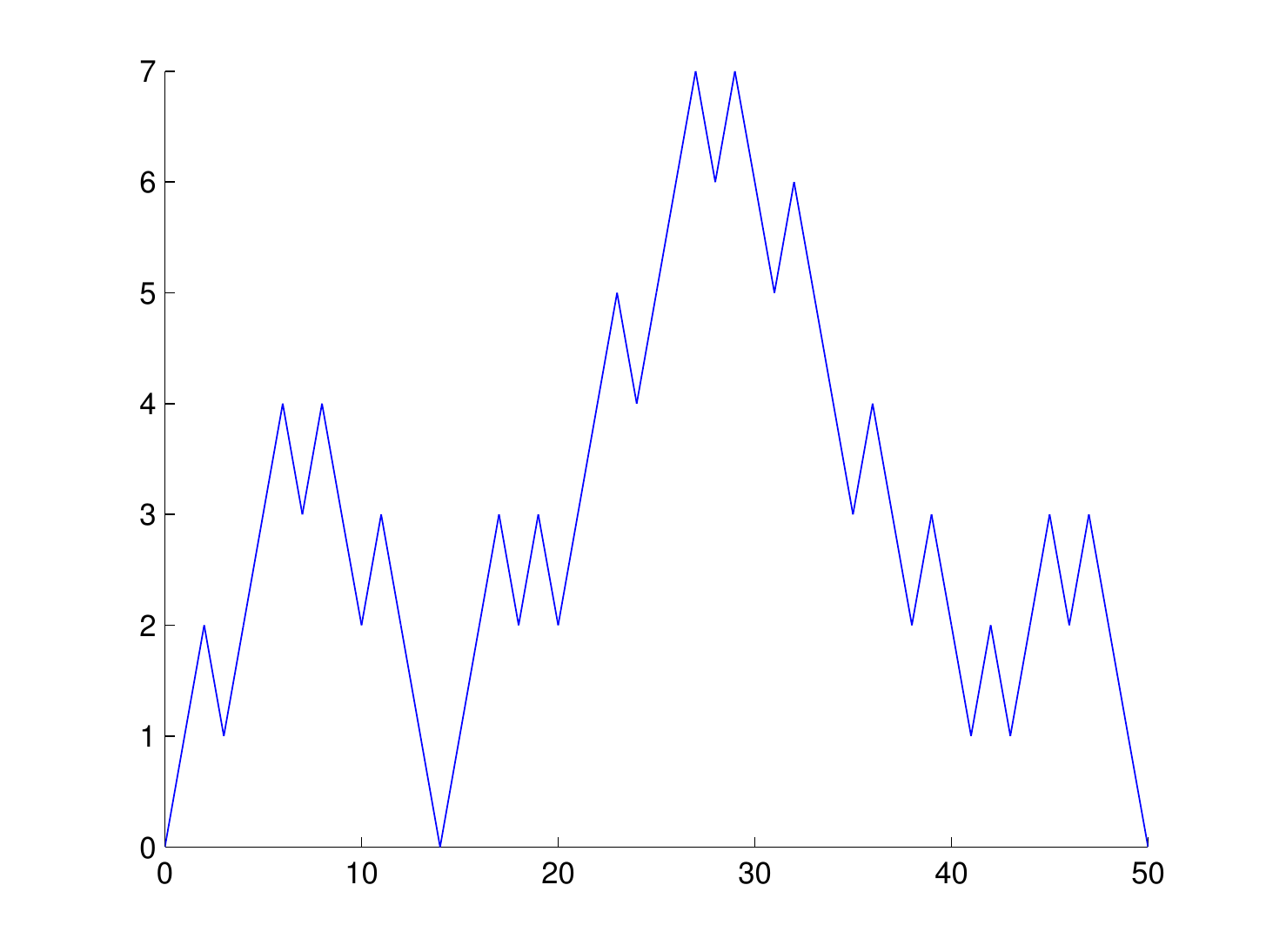}
   \end{minipage}
   \caption{\label{fig:tree1}A tree $\tau$ with its vertices indexed in
lexicographical order and its  contour function $(C_{u}(\tau);\, 0
\leq u \leq 2(\zt-1)$. Here, $|\tau| = 26$.}
   \begin{minipage}[c]{9cm}
   \centering
      \includegraphics[scale=0.4]{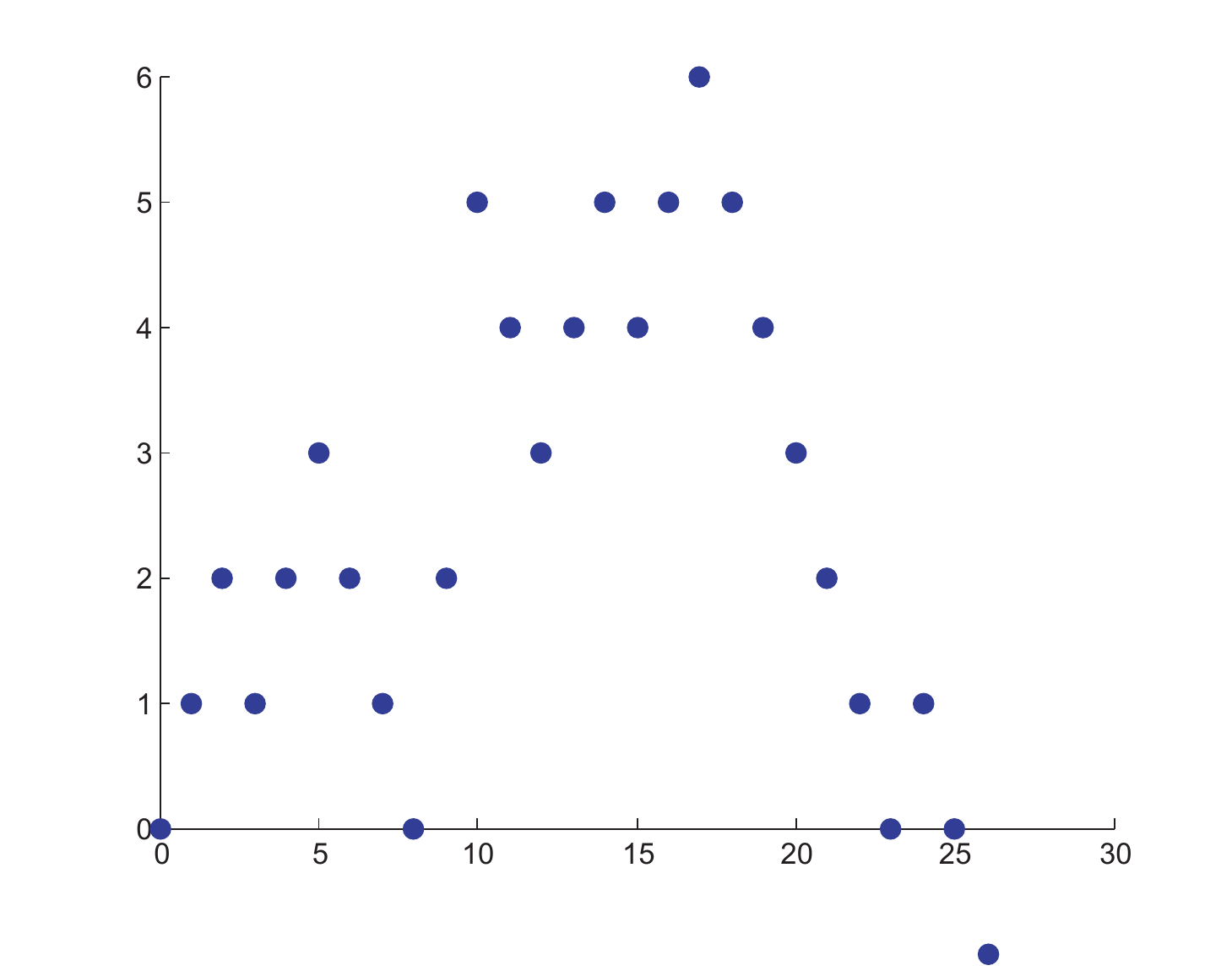}
   \end{minipage}
   \begin{minipage}[c]{9cm}
   \centering
      \includegraphics[scale=0.4]{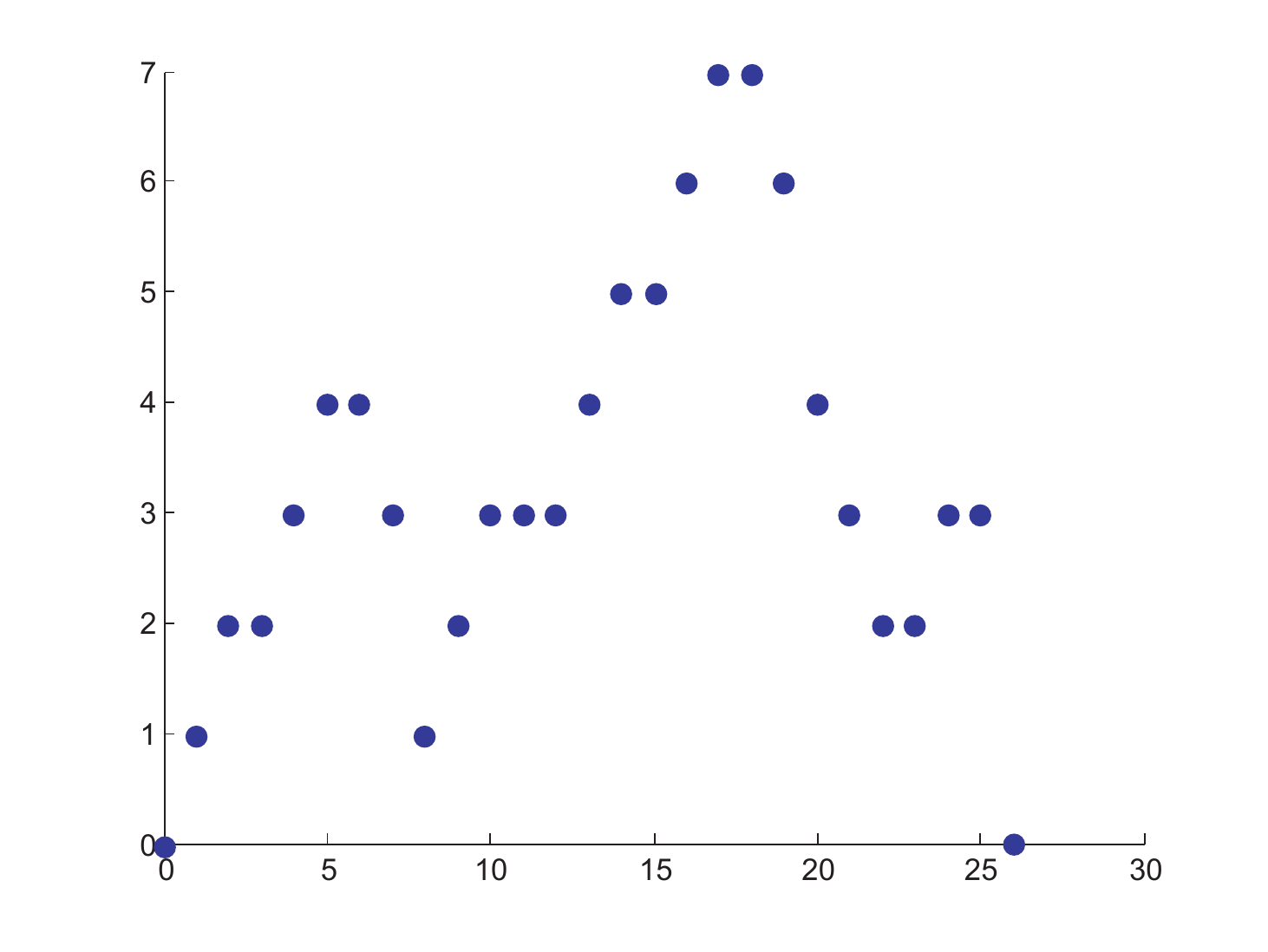}
   \end{minipage}
   \caption{\label{fig:tree} The Lukasiewicz path $ \W(\tau)$ and the height function
$H(\tau)$ of $\tau$.}
\end{figure*}

\begin{defn}\label{def:fonctions}
We write $u < v$ for the lexicographical order on the labels $U$ (for
example $\emptyset<1<21<22$). Let $\tau$ be a finite tree and order the
individuals of $\tau$ in lexicographical order:
$\emptyset=u(0)<u(1)<\cdots<u(|\tau|-1)$. The height process
$H(\tau)=(H_n(\tau), 0 \leq n \leq  |\tau|)$ is defined, for $0
\leq n < |\tau|$, by: $$H_n(\tau)=|u(n)|.$$ We set $H_{|\tau|}(\tau)=0$ for technical reasons. The height $ \H ( \tau)$ of $ \tau$ is by definition $ \max_ {0 \leq n <\zt} H_n( \tau)$.

Consider a particle that starts from the
root and visits continuously all the edges of $\tau$ at unit speed, assuming that every edge has unit length. When
the particle leaves a vertex, it moves toward the first non visited child of this
vertex if there is such a child, or returns to the parent of this
vertex. Since all the edges are crossed twice, the total time needed
to explore the tree is $2 (\zt-1)$. For $0 \leq t \leq
2(\zt-1)$, $C_\tau(t)$ is defined as the distance to the root of the
position of the particle at time $t$. For technical reasons, we set
$C_\tau(t)=0$ for $t \in [2(\zt-1), 2 \zt]$. The function $C(\tau)$
is called the contour function of the tree $\tau$. See Figure
\ref{fig:tree} for an example, and \cite[Section 2]{Du03} for a
rigorous definition.

Finally, the Lukasiewicz path $ \W(\tau)=( \W_n(\tau), 0 \leq n  \leq
|\tau|)$ of $\tau$ is defined by $ \W_0(\tau)=0$ and for
$0 \leq n \leq |\tau|-1$:
$$ \W_{n+1}(\tau)= \W_{n}(\tau)+k_{u(n)}(\tau)-1.$$
\end{defn}
Note that necessarily $ \W_{\zt}(\tau)=-1$ and  that $U( \t_n) =  \min\{ j \geq 0; \, \W_{j+1}(\t_n) - \W_ {j}(\t_n) =  \Delta( \t_{n})-1\}$, where we recall that $U( \t_n)$ is the index in the lexicographical order of the first vertex of $ \t_n$ with maximal out-degree. 
 
The following proposition explains the importance of the Lukasiewicz path. Let $ \rho$ be a critical or subcritical probability distribution  on $ \N$ with $ \rho(1)<1$.

\begin{prop}\label{prop:RW} Let $(W_n)_{n \geq 0}$ be a random walk with starting point $W_0=0$ and jump distribution $\nu(k)=\ \rho(k+1)$ for $k \geq -1$. Set
$\zeta =\inf\{n \geq 0; \, W_n=-1\}$. Then $(W_0,W_1,\ldots,W_{\zeta})$ has the same distribution as the Lukasiewicz path of a
$\GW_\rho$ tree. In particular, the total progeny of a $ \GW_ \rho$ tree has the same law as $\z$.\end{prop}

\begin{proof}See \cite[Proposition 1.5]{LG05}.\end{proof}

We next extend the definition of the Lukasiewicz path to a forest. If $ \textbf{f}= ( \tau_{i})_{1 \leq i \leq  j}$ is a forest, set $n_0=0$ and $n_{p}= | \tau_{1}|+|\tau_{2}|+ \cdots +|\tau_{p}|$ for $1 \leq p  \leq j$. Then, for every $0 \leq i \leq p-1$ and $0 \leq k < |\tau_{i+1}|$, set
$$ \W_{n_{i}+k}( \textbf{f})= \W_{k}( \tau_{i+1})-i.$$
Note that $(\W_{n_{i}+k}( \textbf{f})+i ; 0 \leq k \leq  |\tau_{i+1}| )$ is the Lukasiewicz path of $ \tau_{i+1}$ and that $ \min  \{0 \leq i  \leq n_{j}; \,  \mathcal{W}_{i}( \mathbf{f})=-k\}=n_{k}$ for $ 1 \leq k \leq j$.

Finally, the following result will be useful.
\begin{prop}\label{prop:useful} Let $(W_n)_{ n \geq 0}$ be the random walk introduced in Proposition \ref {prop:RW} with $ \rho= \mu$. Then
\begin{enumerate}
 \item[(i)] $ \Pr { \forall i \geq 1, W_{i} \leq -1}=\gamma$.
 \item[(ii)] For every $i \geq 0$,  $\Pr {\forall m \leq i , W_m \geq 0} = \Prmu { \zt \geq i+1}$.
  \end{enumerate}
\end{prop}

\begin{proof}
By \cite[Theorem 1 in Chapter 2]{Tak77}, we have $ \Pr { \forall i \geq 1, W_{i} \leq -1}=-\Es {W_{1}}= 1-\mathfrak {m}$. The second assertion is an immediate consequence of Proposition \ref {prop:RW}.
\end{proof}

\subsection{The Vervaat transformation}

For $\bx=(x_1, \ldots,x_n)
\in \Z^n$ and $i \in \Znz$, denote by $\bx^{(i)}$ the $i$-th
cyclic shift of $\bx$ defined by $x^{(i)}_k=x_{i+k \mod n}$ for $1
\leq k \leq n$.

\begin{defn}\label{def:vervaat}Let $n \geq 1$ be an integer and let $\bx=(x_1,\ldots,x_{n}) \in
\Z^n$. Set $w_j=x_1+\cdots + x_j$ for $1 \leq j \leq n$ and
let the integer $i_*(\bx)$ be defined by $i_*(\bx) = \inf \{ j \geq 1; w_j= \min_{1 \leq i \leq n} w_i \}$.
The Vervaat transform of $\bx$, denoted by $\V(\bx)$, is defined to
be $\bx^{(i_*(\bx))}$.
\end{defn}

The following fact is well known (see e.g. \cite[Section 5]{Pit06}):

\begin {prop}\label {prop:vervaat}Let $(W_n, n \geq 0)$ be as in Proposition \ref {prop:RW} and $X_k=W_k- W_ {k-1}$ for $k \geq 1$. Fix an integer $n \geq 1$ such that $ \Pr {W_n=-1}>0$. The law of $ \V(X_1, \ldots, X_n)$ under $\Pr{ \, \cdot \, | \, W_{n}=-1}$ coincides with the law of $(X_1, \ldots,X_n)$ under  $\Pr{ \, \cdot \, | \,  \z=n}$.
\end {prop}

From Proposition \ref {prop:RW}, it follows that the law of  $\V(X_1, \ldots, X_ n)$ under $\Pr{ \, \cdot \, | \, W_{n}=-1}$ coincides with the law of $(  \W_1 (\t_n), \W_2 (\t_n)-\W_1 (\t_n), \ldots, \W_n (\t_n)-\W_ {n-1} (\t_n))$ where $ \t_n$ is a $ \GW_ \rho$ tree conditioned on having total progeny equal to $n$.

We now introduce the Vervaat transformation in continuous time.

\begin{defn}Set $\D_0([0,1],\R)= \{ \omega \in \D([0,1],\R) ; \, \omega(0)=0\}$. The Vervaat transformation in continuous time, denoted by $\Vc: \D_0([0,1],\R) \rightarrow
\D([0,1],\R)$, is defined as follows. For $\omega \in
\D_0([0,1],\R)$, set $g_1(\omega)=\inf\{ t \in [0,1]; \omega(t-)
\wedge \omega(t)= \inf_{[0,1]} \omega\}$. Then define:
$$\Vc(\omega)(t)=\begin{cases} \omega(g_1(\omega)+t)-\inf_{[0,1]}
\omega, \qquad \qquad  \qquad \qquad \quad \, \, \, \, \, \textrm{if }
g_1(\omega)+t \leq 1,
\\
\omega(g_1(\omega)+t-1)+\omega(1)-\inf_{[0,1]} \omega \qquad \quad
\qquad \textrm{  if } g_1(\omega)+t \geq 1.
\end{cases}$$\end{defn}

By combining the Vervaat transformation with limit theorems under the conditional probability  distribution $\Pr{ \, \cdot \, | \, W_{n}=-1}$ and using Proposition \ref {prop:RW} we will obtain information about conditioned Galton--Watson trees. The advantage of dealing with $\Pr{ \, \cdot \, | \, W_{n}=-1}$
is to avoid any positivity constraint.

\subsection {Slowly varying functions}

 Recall that a measurable function $L: \R_+ \rightarrow \R_+$ is said to be slowly varying if $ L(x)>0$ for $x$ large enough and $\lim_{x \rightarrow \infty} {L}(tx)/{L}(x)=1$ for all $t>0$. Let $L: \R_+ \rightarrow \R_+$  be a slowly varying function. Without further notice, we will use the following standard facts:
\begin{enumerate}
 \item[(i)] The convergence $\lim_{x \rightarrow \infty} {L}(tx)/{L}(x)=1$ holds uniformly for $t$ in a compact subset of $ (0, \infty)$.
 \item[(ii)] Fix $ \epsilon>0$. There exists a constant $C>1$ such that $\frac{1}{C}x^{-\e}\leq{L(nx)}/{L(n)} \leq Cx^\e$ for every
integer $n$ sufficiently large and $x \geq 1$.  \end{enumerate}
These results are immediate consequences of the so-called representation theorem for slowly varying functions (see e.g. \cite[Theorem 1.3.1]{BGT89}).

\section {Limit theorems for conditioned non-generic Galton--Watson trees}

In the sequel,  $(W_n; n \geq 0)$ denotes the random walk introduced in Proposition \ref {prop:RW} with $ \rho= \mu$. Note that $ \Es {W_1}= -  \gamma<0$. Set $X_0=0$ and $ X_k = W_k - W_ {k-1}$ for $k \geq 1$. It will be convenient to work with centered random walks, so we also set $ \overline {W}_n= W_n+  \gamma n$ and $ \oX_n=X_n + \gamma$ for $n \geq 0$ so that $ \oW_n=\oX_1+ \cdots + \oX_n$. Obviously, $W_n= -1$ if and only if $ \oW_n= \gamma n -1$.

\subsection{Invariance principles for conditioned random walks}

In this section, our goal is to prove  Theorem \ref{thmintro:1}. We first introduce some notation. Denote by $T : \cup_ {n \geq 1}  \R^n  \rightarrow \cup_ {n \geq 1}  \R^n$  the operator that interchanges the last and the (first) maximal component of a finite sequence of real numbers:

$$ T(x_1, \ldots, x_n)_k =  \begin {cases} \d \max_ { 1 \leq i \leq  n} x_i & \textrm {if } k=n \\
x_n & \textrm {if } \d x_k> \max _ {1 \leq i < k} x_i \textrm { and } x_k= \max_ {k \leq  i \leq n} x_i\\
x_k & \textrm {otherwise.}
\end {cases}$$

Since $ \mu$ satisfies Assumption $(H_ \theta)$, we have $  \Pr { \oW_1 \in (x,x+1]} \sim \mathcal {L}(x)/x^ { 1+\theta}$ as $ x \to \infty$. Then, by \cite[Theorem 9.1]{DDS08}, we have:
 \begin{equation}
 \label{eq:estimee}\Pr { \oW_n \in (x,x+1]}   \quad\mathop{ \sim}_{n \rightarrow \infty} \quad  n \cdot \Pr { \oW_1 \in (x,x+1] },
 \end{equation}
 uniformly in $ x \geq  \epsilon n$ for every fixed $ \epsilon >0$. In other words, the distribution of $ \oW_1$ is $(0,1]$--subexponential, so that we can apply a recent result of Armendáriz \& Loulakis \cite{AL11} concerning conditioned random walks with subexponential jump distribution. In our particular case, this result can be stated as follows:

 \begin {thm}[Armendáriz \& Loulakis, Theorem 1 in \cite{AL11}] \label {thm:AL}For $ n \geq 1$ and $x>0$, let $ \mu_ {n,x}$ be the probability measure on $ \R^n$ which is the distribution of $ ( \oX_1, \ldots, \oX_n)$ under the conditional probability distribution $ \Pr { \, \cdot \, | \, \oW_n \in (x,x+1]}$.

 Then for every $ \epsilon>0$, we have:
 $$ \lim _ {n \rightarrow \infty} \sup_ {x \geq  \epsilon n} \sup_ {A \in \mathcal {B}( \R^ {n-1})} \left| \mu_ {n,x} \circ T^ {-1} \left[ A \times \R\right] - \mu^ { \otimes (n-1)} \left[A \right] \right|=0.$$
 \end {thm}
\noindent As explained in \cite{AL11}, this means that under $ \Pr { \, \cdot \, | \, \oW_n \in (x,x+1]}$,  asymptotically one gets $n-1$ independent random variables after forgetting the largest jump.

\bigskip

The proof of Theorem \ref {thmintro:1} is based on the following invariance principle concerning a conditioned random walk with negative drift, which is a simple consequence of Theorem \ref{thm:AL}.

\begin{prop}
\label{prop:invGW}Let $R$ be a uniformly distributed random variable on $[0,1]$. Then the following convergence holds in $ \D([0,1], \R)$:
\begin{equation} \label {eq:invGW}\left( \left. \frac{W_ { \fl {nt}}}{n}, 0 \leq  t \leq 1 \, \right| \, W_{n}=-1\right) \qquad \mathop { \longrightarrow} ^ {(d)}_ { n \to \infty} \qquad (- \gamma  t+ \gamma \mathbbm{1}_ { R \leq t}, 0 \leq t \leq 1).\end{equation}
\end{prop}

\begin {proof}
By the definition of $ \oW$, it is sufficient to check that the following convergence holds in $ \D([0,1], \R)$:
\begin{equation}
\label{eq:cvthm1}
\left( \left. \frac{ \overline{W}_ { \fl {nt}}}{n}, 0 \leq  t \leq 1  \, \right| \, \overline{W}_{n}= \gamma n-1\right) \qquad \mathop { \longrightarrow} ^ {(d)}_ { n \to \infty} \qquad ( \gamma \mathbbm{1}_ { R \leq t}, 0 \leq t \leq 1),
\end{equation}
where $R$ is a uniformly distributed random variable on $[0,1]$. Denote by $V_n$ the coordinate of the first maximal component of $( \oX_1, \ldots, \oX_n)$.  Set $ \widetilde {W}_0=0$ and for $1 \leq i \leq n-1$ set:
$$  \widetilde {W}_i =   \begin {cases} \oX_1+ \oX_2 + \cdots + \oX_i  &  \textrm {if } i< V_n\\
\oX_1+ \oX_2 + \cdots + \oX_ {V_n-1}+ \oX_ {V_n+1}+ \cdots + \oX_ {i+1}   &\textrm {otherwise.}\end {cases}$$
By Theorem \ref {thm:AL}, for every $ \epsilon>0$:
 \begin{equation}
 \label{eq:zero} \lim _ {n \rightarrow \infty} \left| \Pr { \left. \forall t \in[0,1], \, \left| \frac{\widetilde {W}_ { \fl {(n-1) t}}}{n-1}\right| \leq  \epsilon \right| \oW_n = \gamma n -1 } -  \Pr { \forall t \in[0,1], \, \left| \frac{\oW_ { \fl {(n-1) t}}}{n-1}\right| \leq  \epsilon }\right|=0.
 \end{equation}
 Next, by the functional strong law of large numbers,
\begin{equation}
\label{eq:vers0}\Pr { \forall t \in[0,1], \, \left| \frac{\oW_ { \fl {(n-1) t}}}{n-1}\right| \leq  \epsilon }  \quad\mathop{\longrightarrow}_{n \rightarrow \infty} \quad 1.
\end{equation}
Combining \eqref{eq:vers0} with \eqref{eq:zero}, we get the following convergence in $ \D([0,1], \R)$:
\begin{equation}
\label{eq:zero2} \left( \left. \frac{\widetilde {W}_ { \fl {(n-1) t}}}{n-1}; 0 \leq t \leq 1 \, \right| \oW_n = \gamma n -1 \right)  \quad\mathop{\longrightarrow}^ {( \P)}_{n \rightarrow \infty} \quad \mathbf{0},
\end{equation}
where $ \mathbf{0}$ stands for the constant function equal to $0$ on $[0,1]$. In addition, note that on the event $  \{\oW_n = \gamma n -1\}$, we have $ \oX_ {V_n}= \gamma n -1 - \widetilde {W}_ { \fl {n -1}}$. The following joint convergence in distribution thus holds in $ \D([0,1], \R) \times \R$:
\begin{equation}
\label{eq:zero3} \left( \left. \left( \frac{\widetilde {W}_ { \fl {(n-1) t}}}{n-1}; 0 \leq t \leq 1 \right),  \frac{ \oX_ {V_n}}{ n} \, \right| \oW_n = \gamma n -1 \right)  \quad\mathop{\longrightarrow}^ {( \P)}_{n \rightarrow \infty} \quad (\mathbf{0}, \gamma).
\end{equation}
Standard properties of the Skorokhod topology then show that the following convergence holds in $ \D([0,1], \R)$:
\begin{equation}
\label{eq:zero4} \left( \left. \frac{ { \oW}_ { \fl {nt}}}{n} - \frac{ \oX_ {V_n}}{ n} \mathbbm {1}_ {  \{t \geq  \frac{V_n}{n}\}}; 0 \leq t \leq 1 \, \right| \oW_n = \gamma n -1 \right)  \quad\mathop{\longrightarrow}^ {( \P)}_{n \rightarrow \infty} \quad \mathbf{0}.
\end{equation}
Next, note that the convergence  \eqref{eq:zero3} implies that under $ \Pr { \, \cdot \, | \, \oW_n = \gamma n -1}$,  $( \oX_1, \ldots, \oX_n)$ has a unique maximal component with probability tending to one as $n \rightarrow \infty$. Since the distribution of $( \oX_1, \ldots, \oX_n)$ under $ \Pr { \, \cdot \, | \, \oW_n = \gamma n -1}$ is cyclically exchangeable, one easily gets that the law of $V_n/n$ under $ \Pr { \, \cdot \, | \, \oW_n = \gamma n -1}$ converges to the uniform distribution on $[0,1]$. Also from \eqref{eq:zero3} we know that $ \oX_ {V_n}/n$ under $ \Pr { \, \cdot \, | \, \oW_n = \gamma n -1}$ converges in probability to $ \gamma$. It follows that
\begin{equation}
\label{eq:zero5}  \left( \left. \frac{ \oX_ {V_n}}{ n} \mathbbm {1}_ {  \{  \frac{V_n}{n} \leq t\}}, 0 \leq t \leq 1 \right| \,  \oW_n = \gamma n -1 \right)  \quad\mathop{\longrightarrow}^ {(d)}_{n \rightarrow \infty} \quad  ( \gamma \mathbbm{1}_ { R \leq t}, 0 \leq t \leq 1),
\end{equation}
where $R$ is  uniformly distributed over $[0,1]$. Since \eqref{eq:zero4} holds in probability, we can combine \eqref{eq:zero4} and \eqref{eq:zero5} to get \eqref{eq:cvthm1}. This completes the proof. 
\end{proof}

Before proving Theorem \ref {thmintro:1}, we need to introduce some notation.  For $\bx=(x_1, \ldots,x_n)
\in \Z^n$, set $ \mathcal {M}( \bx)= \max_ {1 \leq i \leq n} x_i$. Recall the notation $\V(\bx)$ for the Vervaat transform of $\bx$. Note that  $ \mathcal {M} ( \bx)= \mathcal {M} ( \V(\bx))$. Let $F: \R \rightarrow \R$ be a bounded continuous function. Recall that $ \Delta( \t_n)$ denotes the maximal out-degree of a vertex of $ \t_n$. Since the maximal jump of $ \W( \t_n)$ is equal to $ \Delta( \t_n)-1$, it follows from the remark following Proposition \ref {prop:vervaat} that:
\begin{eqnarray}
\Es{ F( \Delta( \t_n))} &=& \Es{ F \left(  \mathcal {M} \left(\V(X_1,X_2, \ldots, X_n\right) \right)+1) \, | \, W_n=-1} \notag \\
&=&  \Es{ F \left(  \mathcal {M} \left(X_1,X_2, \ldots, X_n\right) +1\right) \, | \, W_n=-1} \label{eq:eg}
\end{eqnarray}

Recall that since $ \mu$ satisfies Assumption $ (H_ \theta)$, $  \oW_1$ belongs to the domain of attraction of a spectrally positive strictly stable law of index $2 \wedge \theta$. Hence there exists a slowly varying function $L$ such that $ \oW_n/ \left(L(n) n^ { 1/ ( 2 \wedge \theta)} \right)$ converges in distribution toward $Y_1$. We set $ B_n=L(n) n^ { 1/ ( 2 \wedge \theta)}$ and prove that Theorem \ref {thmintro:1} holds with this choice of $B_n$. The function $L$ is not unique, but if $ \widetilde {L}$ is another slowly function with the same property we have $L(n)/ \widetilde {L}(n) \rightarrow 1$ as $n \rightarrow \infty$. So our results do not depend on the choice of $L$. Note that when $ \mu$ has finite variance $ \sigma^2$, one may take $B_n= \sigma \sqrt {n/2}$, and when $ \mathcal {L}=c+o(1)$ one may choose $L$ to be a constant function.

We are now ready to prove Theorem \ref {thmintro:1}.

\begin{proof}[Proof of Theorem \ref {thmintro:1}] If $ Z \in \D([0,1], \R)$, denote by $ \overline { \Delta}(Z)= \sup_{0<s<1}(Z_{s}-Z_{s-})$ the largest jump of $Z$. Since $ \overline{\Delta} : \D([0,1], \R) \rightarrow \R$ is continuous, from Proposition  \ref{prop:invGW} we get that, under the conditioned probability measure $ \Pr{ \, \cdot \, | W_{n}=-1}$, $  \mathcal {M} \left(X_1,X_2, \ldots, X_n\right)/n$ converges in probability toward $ \gamma$ as $n \rightarrow \infty$. Assertion (i) immediately follows from \eqref{eq:eg}.

For the second assertion, keeping the notation of the proof of Proposition \ref {prop:invGW}, we get by Theorem \ref{thm:AL} that for every bounded continuous function $F : \D([0,1], \R) \rightarrow \R$
$$\lim _ {n \rightarrow \infty} \left| \Es { \left. F \left( \frac{\widetilde {W}_ { \fl {(n-1)t}}}{B_{n}}; 0 \leq t \leq 1\right)   \right| \oW_n = \gamma n -1 } -  \Es {  F(Y_{t}, 0 \leq t \leq 1)  }\right|=0.$$
Since the jumps of $ \left(\widetilde {W}_ { \fl {(n-1)t}}; 0 \leq t \leq 1\right)$ have the same distribution as the out-degrees, minus one, of all the vertices of $ \t_n$, except $ u_{ \star}( \t_{n})$, and by continuity of the map $ Z \mapsto \overline { \Delta}(Z)$  on $ \D([0,1], \R)$, it follows that 
$$ \frac{D_{n}}{B_{n}}  \quad\mathop{\longrightarrow}^{(d)}_{n \rightarrow \infty} \quad \overline { \Delta}(Y_t, 0 \leq t \leq 1).$$
If $ \theta \geq 2$, $Y$ is continuous and $\overline { \Delta}(Y_t, 0 \leq t \leq 1)=0$. If $ \theta<2$, the result easily follows from the fact that the Lévy measure of $Y$ is $ \nu(dx)= \mathbbm {1}_ {  \{x >0 \}}dx/( \Gamma(- \theta)x^ {1+ \theta})$.

For (iii), if $V_n$ is as in the proof of Proposition \ref {prop:invGW}, note that we have $$\mathcal{M}( \oX_1, \oX_2, \ldots, \oX_n)= \oX_ {V_n}= \gamma n -1 - \sum_ {i \neq V_n} \oX_i= \gamma n -1 - \widetilde {W}_{n-1}$$ on the event $  \left\{ \oW_n= \gamma n  -1 \right\}$. As noted in \cite[Formula (2.7)]{AL11}, it follows from \eqref{eq:zero2} that
$$ \left( \left. \frac{\mathcal{M}( \oX_1, \oX_2, \ldots, \oX_n) - \gamma n }{B_n} \, \right | \, \oW_n=  \gamma n- 1 \right) \quad\mathop{\longrightarrow} ^ {(d)}_{n \rightarrow \infty} \quad -Y_1.$$
Since $\mathcal{M}( \oX_1, \oX_2, \ldots, \oX_n)=\mathcal{M}( X_1, X_2, \ldots, X_n) + \gamma$, we thus get that $$ \left( \left. \frac{ \mathcal {M} \left(X_1,X_2, \ldots, X_n\right) +1 - \gamma n }{B_n }  \, \right | \, W_n= - 1 \right) \quad\mathop{\longrightarrow} ^ {(d)}_{n \rightarrow \infty} \quad -Y_1.$$
Assertion (iii) then immediately follows from \eqref{eq:eg}. This completes the proof.
\end{proof}

\begin {rem} \label{rem:gen}The preceding proof shows that assertion (i) in Theorem \ref {thmintro:1} remain true when $\mu$ is subcritical and both \eqref{eq:estimee} and Theorem \ref{thm:AL} hold. These conditions are more general than those of Assumption $(H_ \theta)$: see e.g.  \cite[Section 9]{DDS08} for examples of probability distributions that do not satisfy Assumption $(H_ \theta)$ but such that \eqref{eq:estimee} holds. Note that assertion (ii) in Theorem \ref {thmintro:1}  relies on the fact that $ \mu$ belongs to the domain of attraction of a stable law. Note also that there exist subcritical probability distributions such that none of the assertions of Theorem \ref {thmintro:1}  hold (see \cite[Example 19.37]{Jan12} for an example).
\end {rem}

By applying the Vervaat transformation in continuous time to the convergence of Proposition \ref{prop:invGW}, standard properties of the Skorokhod topology imply the following invariance principle for the Lukasiewicz path coding $ \t_{n}$ (we leave details to the reader since we will not need this result later). See Fig.~\ref{fig:luka} for a simulation.

 \begin{figure*}[h!]
\begin{center}
\includegraphics[scale=0.5]{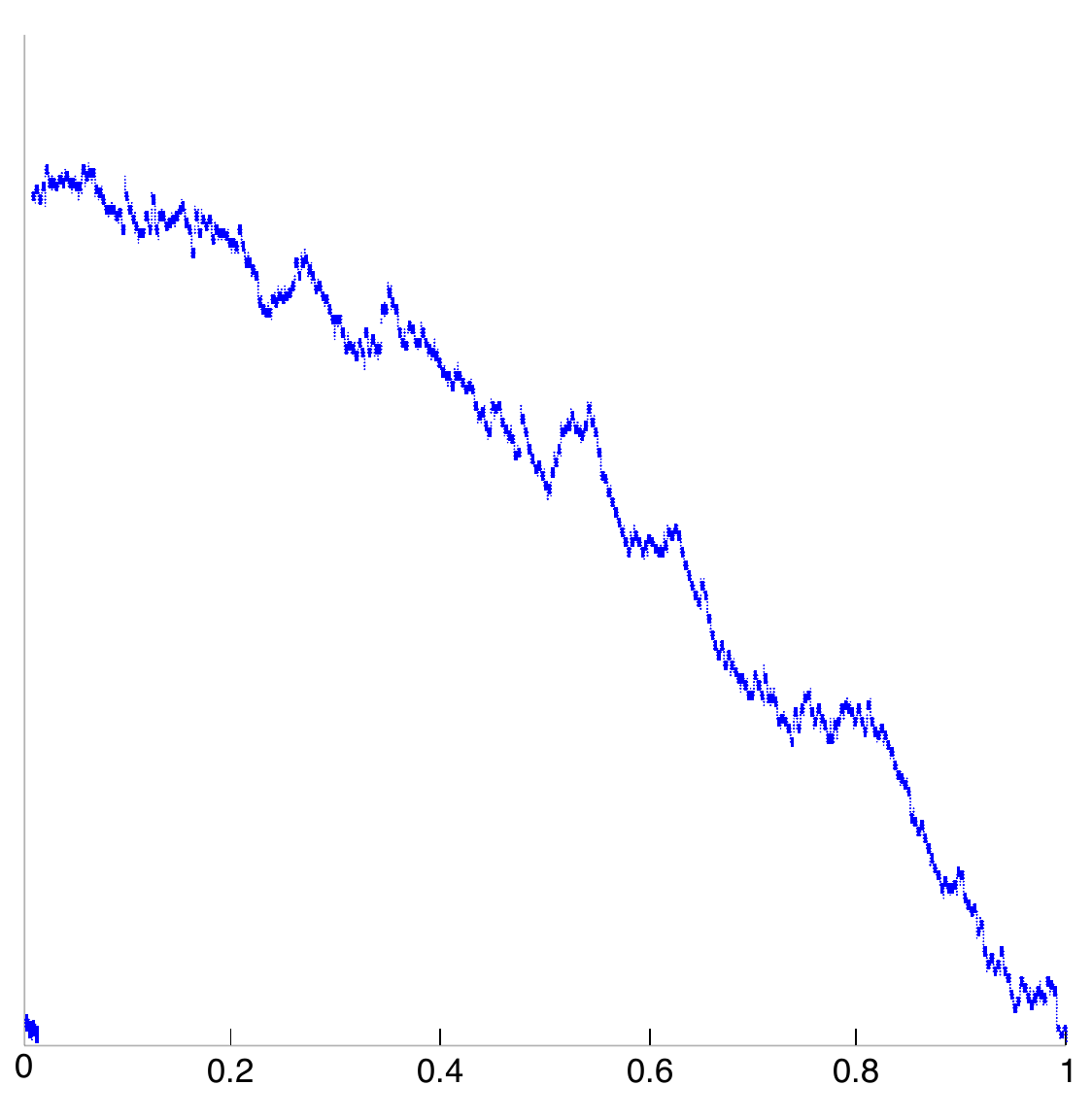}
\caption{\label{fig:luka}A simulation of a Lukasiewicz path of a large non-generic tree.}
\end{center}
\end{figure*}

\begin {prop} \label {prop:luka}The following assertions hold.
\begin{enumerate}
 \item[(i)] We have:
 $$ \sup_ { 0 \leq i \leq U( \t_n)} \frac{ \W_i ( \t_n)}{n}  \quad\mathop{\longrightarrow}^ {( \P)}_{n \rightarrow \infty} \quad 0.$$
 \item[(ii)] The following convergence holds in distribution in $ \D([ 0,1] , \R)$: $$ \d \left( \frac{\W_ { \fl {nt} \vee (U ( \t_n)+1)}( \t_n) }{n}, 0 \leq t \leq 1 \right) \quad \mathop { \longrightarrow}^ {(d)}_ {n \to \infty} \quad \left( \gamma (1-t), 0 \leq t \leq 1 \right).$$
 \end{enumerate}
 \end {prop}

\noindent   Property (i) shows that $\left( \W_ { \fl {nt}}( \t_n) /{n}, 0 \leq t \leq 1 \right)$ does not converge in distribution in $ \D([ 0,1] , \R)$ toward $\left( \gamma (1-t), 0 \leq t \leq 1 \right)$ and this explains why we look at the Lukasiewicz path only after time $U( \t_n)$ in (ii).

\subsection{Description of the Lukasiewicz path after removing the vertex of maximal degree}
\label{sec:descr}

Recall that $U( \tau)$ is the index in the lexicographical order of the first vertex of $ \tau$ with maximal out-degree. 
We first define a modified version $ \widetilde{ \W}( \tau)$ of the Lukasiewicz path as follows. Set $n=|\tau|$, and for $1 \leq i \leq n-U( \tau)-1$, set $ \widetilde{ \mathcal{X}}_{i}( \tau)= \W_{U( \tau)+i+1}( \tau)-\W_{U( \tau)+i}( \tau)$ and  for $n-U( \tau) \leq i \leq n-1$ set 
$\widetilde{ \mathcal{X}}_{i}( \tau)= \W_{i+1-(n-U( \tau))}( \tau)-\W_{i-(n-U( \tau))}( \tau)$.  In other words, $\widetilde{ \mathcal{X}}_{1}( \tau), \ldots, \widetilde{ \mathcal{X}}_{n-1}( \tau)$ are the increments of $ \W(\t_n)$, shifted cyclically to start just after the maximum jump (which is not included).
Then set $\widetilde{ \W}_{i}( \tau)= \widetilde{ \mathcal{X}}_{1}( \tau)+\widetilde{ \mathcal{X}}_{2}( \tau)+ \cdots +\widetilde{ \mathcal{X}}_{i}( \tau)$  for $0 \leq i \leq n-1$ (see Fig.~\ref{fig:Wtilde} for an example). Note that $\Delta(\tau)=-\widetilde{ \mathcal{W}}_{n-1}(\tau)$. Finally, set
$$I( \tau)= \min  \{i  \in   \{0,1,\ldots,n-1\}; \quad  \widetilde{ \W}_{i}( \tau)=  \min_{0 \leq j \leq n-1}\widetilde{ \W}_{j}( \tau) \}.$$
\begin{figure*}[h]
\begin{center}
      \includegraphics[scale=0.5]{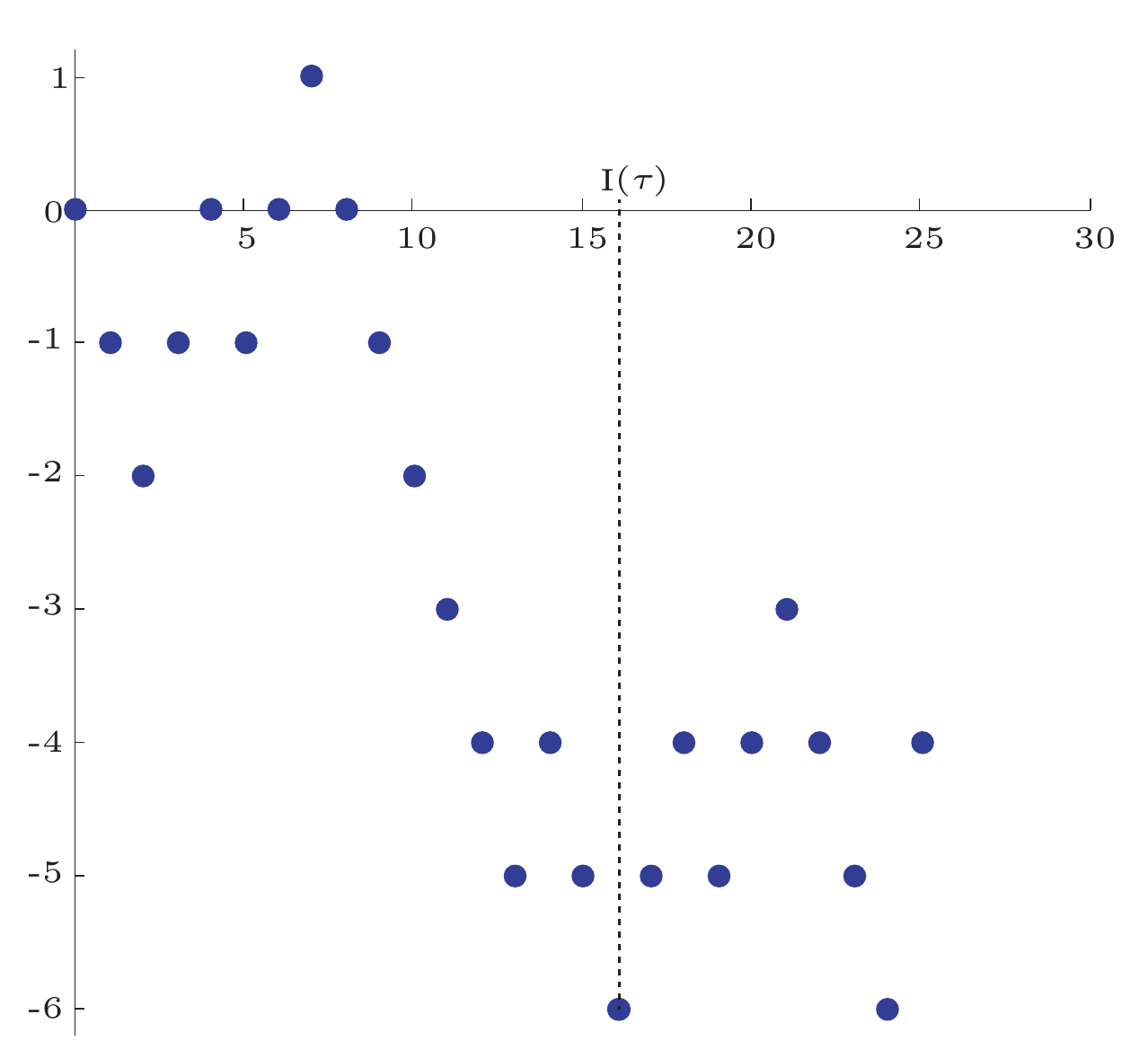}
\caption{\label{fig:Wtilde} The modified Lukasiewicz path $\widetilde{ \W}( \tau)$ of the tree $ \tau$ appearing in Fig.~\ref{fig:tree1}. Here $n=26$, $\Delta(\tau)=4$, $I(\tau)=16$ and $U( \tau)=9$.}
\end{center}
\end{figure*}

We now introduce some notation. For every $1 \leq i \leq \Delta( \tau)$, let $\mathcal{T}_{i}( \tau)$ be the tree of descendants of the $i$-th child of $ u_{ \star}( \tau)$. Set also $\mathcal{F}_{i,j}(\tau)=(\mathcal{T}_{i}( \tau), \ldots, \mathcal{T}_{j}( \tau))$. Finally, for $1 \leq k \leq  \Delta(\tau)$, let $ \widetilde{\z}_{k}(\tau)= \inf  \{i \geq 0; \widetilde{ \mathcal{W}}_{i}(\tau)=-k\}$. 
The following result explains the reason why we introduce $ \widetilde{ \W}( \tau)$. 

\begin {prop}\label {prop:links}The following assertions hold.
\begin{enumerate}
 \item[(i)] We have $U( \tau)=n-1-I( \tau)$.
  \item[(ii)] For $1 \leq k \leq  \Delta( \tau)$, $(\widetilde{ \mathcal{W}}_{0},\widetilde{ \mathcal{W}}_{1}, \ldots,\widetilde{ \mathcal{W}}_{ \widetilde{\z}_{k}(\tau)})$ is the Lukasiewicz path of the forest $ \mathcal{F}_{1,k}( \tau)$.
  \item[(iii)]The vectors $( \W_{0}( \tau), \W_{1}( \tau), \ldots, \W_{U( \tau)}( \tau))$ and
  $$(\widetilde{ \W}_{ I( \tau)}( \tau)-\widetilde{ \W}_{ I( \tau)}( \tau),\widetilde{ \W}_{ I( \tau)+1}( \tau)-\widetilde{ \W}_{ I( \tau)}( \tau), \ldots, \widetilde{ \W}_{ n-1}( \tau)-\widetilde{ \W}_{ I( \tau)}( \tau))$$
  are equal.
   \end{enumerate}
\end {prop}

This should be clear from the relation between $ \mathcal{W}$ and $\widetilde{ \W}$, see Fig.~\ref{fig:Wtilde}), and a formal proof would not be enlightning. We now prove that the random variables $ \widetilde{ \mathcal{X}}_{i}( \t_{n})$ are asymptotically independent.

\begin{prop} \label{prop:tilde}We have:
$$ \sup_{A \in \mathcal{B}( \R^{n-1})} \left| \Pr{(\widetilde{ \mathcal{X}}_{1}( \t_{n}), \ldots, \widetilde{ \mathcal{X}}_{n-1}) \in A } - \Pr{(X_{1}, \ldots, X_{n-1}) \in A}  \right| \quad\mathop{\longrightarrow}_{n \rightarrow \infty} \quad 0.$$
\end{prop}

\begin{proof}We keep the notation introduced in the proof of Proposition \ref{prop:invGW}. By Proposition \ref{prop:vervaat} and by the definition of $ \widetilde{ \W}( \t_{n})$, we have 
$$(\widetilde{ \mathcal{X}}_{i}( \t_{n}); 0 \leq i \leq n-1) \quad \mathop{=}^{(d)}  \quad (X_{V_{n}+1}, X_{V_{n}+2}, \ldots, X_{n},X_{1}, \ldots,X_{V_{n}-1}) \quad \textrm{ under } \Pr{ \, \cdot \, | W_{n}=-1}.$$  On the event that $( X_1, \ldots, X_n)$ has a unique maximal component, we have $$(X_{V_{n}+1}, X_{V_{n}+2}, \ldots, X_{n},X_{1}, \ldots,X_{V_{n}-1})  \quad \mathop{=}^{(d)} \quad (X_{1},X_{2}, \ldots, X_{V_{n}-1},X_{V_{n+1}},  \ldots X_{n}).$$
Indeed, since the distribution of $( X_1, \ldots, X_n)$ under $ \Pr { \, \cdot \, | \, W_n = -1}$ is cyclically exchangeable, $V_{n}$ is uniformly distributed on the latter event. 
But we have already seen that under $ \Pr { \, \cdot \, | \, W_n =  -1}$,  $( X_1, \ldots, X_n)$ has a unique maximal component with probability tending to one as $n \rightarrow \infty$. The conclusion follows from Theorem \ref{thm:AL}.
\end{proof}

The following corollary will be useful.
\begin{cor}\label{cor:indep} Fix $ \eta \in (0,1)$.
We have
$$ \sup_{A \in \mathcal{B}( \mathbb{T}_{f}^{ \fl{  \gamma \eta n}})} \left| \Pr{( \mathcal{T}_{1}( \t_{n}), \ldots, \mathcal{T}_{ \fl{  \gamma \eta n}}( \t_{n}) ) \in A} - \mathbb{P}_{\mu}^{ \otimes \fl{  \gamma \eta n}} (A) \right|  \quad\mathop{\longrightarrow}_{n \rightarrow \infty} \quad 0.$$
\end{cor}
This means that, as $n \rightarrow \infty$, the random variables  $\mathcal{T}_{1}( \t_{n}), \ldots, \mathcal{T}_{ \fl{  \gamma \eta n}}( \t_{n})$ are asymptotically independent $ \GW_{ \mu}$ trees.
\begin{proof} Since $\inf \{ i \geq 0, W_i = -\fl{  \gamma \eta n}\}/n$ converges in probability towards $\eta$ as $n \to \infty$, Proposition \ref{prop:tilde} entails that $\widetilde{\z}_{\fl{  \gamma \eta n}}( \t_{n})/n \rightarrow \eta$ in probability as $ n \rightarrow \infty$. This implies that for every $ \epsilon \in (0,1- \eta)$, with probability tending to $1$ as $ n \rightarrow \infty$, $\mathcal{T}_{1}( \t_{n}), \ldots, \mathcal{T}_{ \fl{  \gamma \eta n}}( \t_{n})$ only depends on $\widetilde{ \mathcal{X}}_{1}( \t_{n}), \ldots,\widetilde{ \mathcal{X}}_{ \fl{ \gamma (\eta+ \epsilon)n}}( \t_{n})$. The conclusion immediately follows from Proposition \ref{prop:tilde}.
\end{proof}

\subsection{Location of the vertex with maximal out-degree}
\label {section:location}

The main tool for studying the modified Lukasiewicz path $ \widetilde{ \W}$ is a time-reversal procedure which we now describe. For a sequence $(a_{i})_{i \geq 0}$ and for every integer $n \geq 1$, we let  $(a^{(n)}_{i})_{0 \leq i \leq n-1}$ be the sequence defined by $a^{(n)}_{i}=a_{n-1}-a_{n-1-i}$.

\begin {proof}[Proof of Theorem \ref {thm:cvd} (i)]  For every tree $\tau$, writing $\widetilde{ \W}$ instead of $\widetilde{ \W}(\tau)$ to simplify notation, note that by Proposition \ref{prop:links} (i) we have $$U( \tau)= \max \left\{ 0 \leq k \leq n-1; \quad\widetilde{ \W}^{(n)}_{k}=  \sup_{0 \leq j \leq n-1}  \widetilde{ \W}^{(n)}_{j} \right\}.$$
It follows from Proposition \ref{prop:links} (i) and Proposition \ref{prop:tilde} that, for every $i \geq 0$,
$$ \Pr{ U( \t_{n})=i}- \Pr{ i =  \max \{ 0 \leq k \leq  n-1; \quad { W}^{(n)}_{k} = \sup_{0 \leq j \leq n-1}  { W}^{(n)}_{j} \}} \quad\mathop{\longrightarrow}_{n \rightarrow \infty} \quad  0.$$
Since $(W^{(n)}_{i})_{0 \leq i \leq n-1}$ and $(W_{i})_{0 \leq i \leq n-1}$ have the same distribution, we have $$\max \{ 0 \leq k \leq  n-1; \quad { W}^{(n)}_{k} = \sup_{0 \leq j \leq n-1}  { W}^{(n)}_{j} \}  \quad \mathop{=} ^{(d)}  \quad  \max \{ 0 \leq k \leq  n-1; \quad { W}_{k} = \sup_{0 \leq j \leq n-1}  { W}_{j} \}.$$ In addition, $W$ has a negative drift and tends almost surely to $- \infty$, hence
$$ \Pr{ i= \max \{ 0 \leq k \leq  n-1; \quad { W}_{k} = \sup_{0 \leq j \leq n-1}  { W}_{j} \}}  \quad\mathop{\longrightarrow}_{n \rightarrow \infty} \quad  \Pr{ i= \max \{ k \geq 0; \quad { W}_{k} = \sup_{j \geq 0}  { W}_{j} \}} .$$
A simple argument based once again on time-reversal shows that the probability appearing in the right-hand side of the previous expression is equal to $$\Pr {\forall m \leq i , W_m \geq 0} \cdot \Pr { \forall j \geq 1, W_{j} \leq -1}.$$
Assertion (i) in Theorem \ref {thm:cvd} then follows from Proposition \ref{prop:useful}.
\end{proof}

\begin{rem}\label{rem:cl}One similarly shows that
\begin{equation}
\label{eq:cl}n-1-\widetilde{\z}_{\Delta(\t_{n})}(\t_n)  \quad\mathop{\longrightarrow}^{(d)}_{n \rightarrow \infty} \quad  \sup\{ i \geq 0; \quad W_{i}=0\}.
\end{equation}
Indeed, since $ \Delta( \t_{n})=-\widetilde{ \W}_{n-1}(\t_{n})$, we have $n-1-\widetilde{\z}_{\Delta(\t_{n})}(\t_n)=\max \left\{ 0 \leq i \leq n-1; \quad\widetilde{ \W}^{(n)}_{i}=  0 \right\}$, and \eqref{eq:cl} follows by the same arguments as those in the proof of Theorem \ref {thm:cvd} (i).
\end{rem}

To prove the other assertions of Theorem \ref {thm:cvd}, we will need the size-biased distribution associated with $ \mu$, which is the distribution of the random variable $  { \zeta}^*$ such that:
$$ \Pr {{ \zeta^*}=k} :=  \frac{k \mu_k}{ \mathfrak {m}} \qquad k=0,1, \ldots.$$
The following result concerning the local convergence of $ \t_n$ as $n \rightarrow \infty$ will be useful. We refer the reader to \cite[Section 6]{Jan12} for definitions and background concerning local convergence of trees (note that we need to consider trees that are not locally finite, so that this is slightly different from the usual setting).

Let $ \widehat{\mathcal {T}}$ be the infinite random tree constructed as follows. Start with a spine composed of a random number $S$ of vertices, where $S$ is defined by:
\begin{equation}
\label{eq:spine}\Pr {S=i}= (1-\mathfrak {m}) \mathfrak {m} ^ {i-1}, \qquad i=1,2, \ldots.
\end{equation} Then attach further branches as follows (see also Figure \ref {fig:GW} below). At the top of the spine, attach an infinite number of branches, each branch being a $ \GW_ \mu$ tree. At all the other vertices of the spine, a random number of branches distributed as $ \zeta^*-1$ is attached to  either to the left or to the right of the spine, each branch being a $ \GW_ \mu$ tree. At a vertex of the spine where $k$ new branches are attached, the number of new branches attached to the left of the spine is uniformly distributed on $  \{0, \ldots,k\}$. Moreover all random choices are independent.

\begin{thm}[Jonsson \& Stef\'ansson  \cite{JS11}, Janson \cite {Jan12}]\label {thm:local}
The trees  $ \t_n$ converges locally in distribution toward $ \widehat{\mathcal {T}}$ as $ n \rightarrow \infty$.
\begin{figure*}[h]
\begin{center}
\includegraphics[scale=0.8]{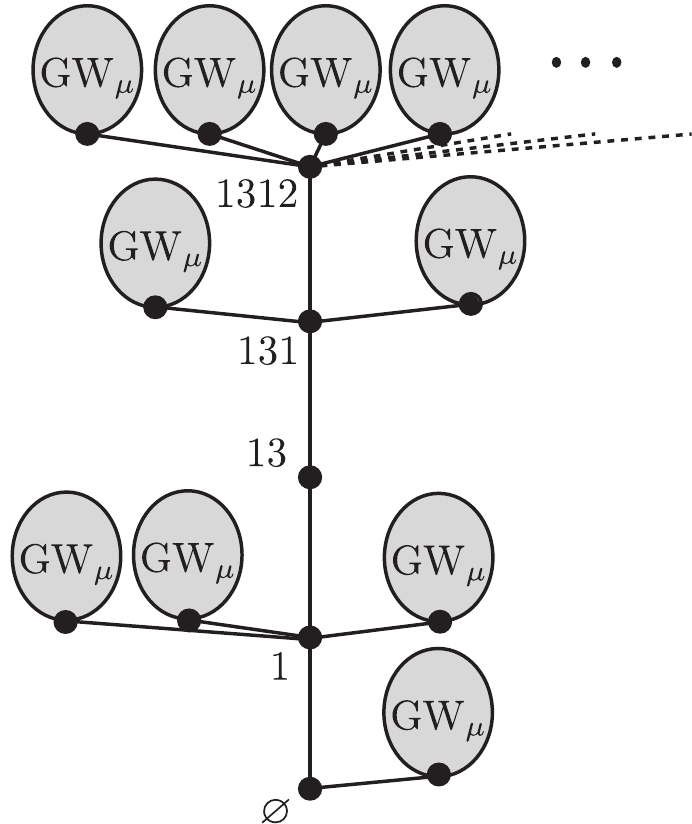}
\caption{\label{fig:GW}An illustration of $ \widehat{\mathcal {T}}$. Here, the spine is composed of the vertices $ \emptyset,1,13,13,131,1312$.}
\end{center}
\end{figure*}
\end {thm}

\begin {proof}[Proof of Theorem \ref {thm:cvd} (ii) and (iii)]By Skorokhod's representation
theorem (see e.g. \cite[Theorem 6.7]{Bil99}) we can suppose that the convergence $\t_n \rightarrow \widehat{\mathcal {T}}$ as $n \rightarrow  \infty$  holds almost surely for the local topology. Let $u_ \star \in \widehat{\mathcal {T}}$ be the vertex of the spine with largest generation. By \eqref{eq:spine}, we have for $i \geq 0$:
\begin{equation}
\label{eq:hspine} \Pr {| u_ \star| =i}=(1-\mathfrak {m}) \mathfrak {m} ^ {i}.
\end{equation}
Recall the notation $U( \t_n)$ for the index of $u_ \star( \t_n)$. Let $ \epsilon>0$. By assertion (i) of Theorem \ref {thm:cvd}, which was proved at the beginning of this section, we can fix an integer $K$ such that, for every $n$, $ \Pr {U ( \t_n) \leq K} >1- \epsilon$. From the local convergence of $ \t_n$ to $\widehat{\mathcal {T}}$ (and the properties of local convergence, see in particular Lemma 6.3 in \cite {Jan12}) we can easily verify that
$$ 	\Pr {  \{ u_ \star ( \t_n) \neq u_ \star\} \cap  \{ U( \t_n) \leq K\}}  \quad\mathop{\longrightarrow}_{n \rightarrow \infty} \quad 0.$$
We conclude that $ \Pr {u_ \star ( \t_n) \neq u_ \star} \rightarrow 0$ as $ n \rightarrow \infty$. Assertion (ii) of Theorem \ref{thm:cvd} now follows from \eqref{eq:hspine}.\end {proof}

Note that assertion (i) in Theorem \ref {thm:cvd} was needed to prove assertion (ii). Indeed, the local convergence of $ \t_n$ toward $\widehat{\mathcal {T}}$ would not have been sufficient to get that $ \Pr {u_ \star ( \t_n) \neq u_ \star} \rightarrow 0$.

\subsection{Subtrees branching off the vertex with maximum out-degree}
\label{sec:branching}

Before proving   Theorem \ref {thm:fluctenfants}, we gather a few useful ingredients.  It is well known that the mean number of vertices of a $ \GW_ \mu$ tree at generation $n$ is $ \mathfrak {m}^n$. As a consequence, we have $ \Esmu { |\tau|}=1+\mathfrak {m}+\mathfrak {m}^2+\cdots=1/(1-\mathfrak {m})= 1/ \gamma$. Moreover, for $ n \geq 1$, by Kemperman's formula (see e.g. \cite[Section 5] {Pit06}):
\begin{equation}
\label{eq:zn}\Prmu { |\tau|=n}= \frac{1}{n} \Pr {W_n=-1}=\frac{1}{n} \Pr {\oW_n= \gamma n-1} \quad \mathop { \sim}_ {n \rightarrow \infty} \quad \frac{ \mathcal {L}(n)}{ ( \gamma n) ^ {1+ \theta}},
\end{equation}
where we have used \eqref{eq:estimee} for the last estimate. It follows that the total progeny of a $ \GW_ \mu$ tree belongs to the domain of attraction of a spectrally positive strictly stable law of index $2 \wedge \theta$. Hence we can find a slowly varying function $L'$ such that the law of $ (\z ( \bf)-n/ \gamma)/ \left( L'(n) n^ { 1/(2 \wedge \theta)} \right)$ under $ \mathbb {P}_ { \mu,n}$ converges as $n \rightarrow \infty$ to the law of $Y_1$, where we recall that $ \Pmuj$ is the law of a forest of $j$ independent $ \GW_ \mu$ trees. We set $B'_n= L'(n) n^ { 1/(2 \wedge \theta)}$.

 Let $( \cZ_i)_ {i \geq 0}$ be the random walk which starts at $0$ and whose jump distribution has the same law as the total progeny of a $ \GW_ \mu$ tree. Note that $\Pr{ \cZ_j=k}= \Prmuj { \z( \bf)=k}$. Hence the distribution of $ \cZ_1$  belongs to the domain of attraction of a spectrally positive strictly stable law of index $2 \wedge \theta$. In particular, the following convergence holds in distribution in the space $ \D([0,1], \R)$: \begin{equation}
\label{eq:cvZ} \left(\frac{ \cZ_ { \fl {nt}} - nt/ \gamma}{B'_n}, 0 \leq t \leq 1 \right) \quad \mathop { \longrightarrow}^ {(d)}_ {n \rightarrow \infty} \quad (Y_t, 0 \leq t \leq 1).
\end{equation}

Finally, the following technical result establishes a useful link between $B_n$ and $B'_n$.

\begin {lem}\label {lem:link} We have $ {B'_n}/{B_n} \rightarrow 1/ \gamma^ {1+1/(2 \wedge \theta)}$ as $ n \rightarrow \infty$.
\end {lem}

\noindent The proof of Lemma \ref {lem:link} is postponed to the end of this section.

We are now ready to prove Theorem \ref {thm:fluctenfants}.

\begin {proof}[Proof of Theorem \ref {thm:fluctenfants}] We shall show that for every fixed $ \eta \in (0,1)$:

\begin{equation}
\label{eq:amontrer} \left( \frac{Z_ { \fl {  \Delta( \t_n) t}}( \t_n)-  \Delta( \t_n) t/ \gamma}{B_ { n}}, 0 \leq t \leq \eta\right) \quad \mathop { \longrightarrow}^ {(d)}_ { n \rightarrow \infty} \quad
\d( \frac{1}{ \gamma} Y_t, 0 \leq t \leq \eta).
\end{equation}
Since $ \Delta( \t_{n})/( \gamma n) \rightarrow 1$ in probability as $n \rightarrow \infty$ (Theorem \ref{thmintro:1} (i)), recalling Lemma \ref {lem:link}, the desired result will follow from a time-reversal argument since the vectors $( \xi_1( \t_n), \xi_2( \t_n), \ldots,\xi_ { \Delta( \t_n)}( \t_n))$ and $( \xi_ { \Delta( \t_n)}( \t_n), \xi_ { \Delta( \t_n)-1}( \t_n), \ldots, \xi_1( \t_n))$ have the same distribution. Tightness follows from the time-reversal argument and also continuity at $t=1$.

Since $ \Delta( \t_{n})/( \gamma n) \rightarrow 1$ in probability as $n \rightarrow \infty$,  by \cite[Lemma 5.7]{Kor12}, it is sufficient to establish that
\begin{equation}
\label{eq:amontrer2} \left( \frac{Z_ { \fl { \gamma n t}}( \t_n)-  n t}{B_ {  n}}, 0 \leq t \leq \eta\right) \quad \mathop { \longrightarrow}^ {(d)}_ { n \rightarrow \infty} \quad
\d(\frac{1}{ \gamma}  Y_t, 0 \leq t \leq \eta).
\end{equation}
To this end, note that $B'_{\fl{\gamma n}}/B_n\to 1/\gamma$ as $n \to \infty$ and that $(Z_{1}( \t_{n}), Z_{2}( \t_{n}), \ldots, Z_{  \fl { \gamma n  \eta}}( \t_{n}))$ are asymptotically independent by Corollary \ref{cor:indep}. The conclusion immediately follows from \eqref{eq:cvZ} applied with $\fl{\gamma n}$ instead of $n$.
\end{proof}

We conclude this section by proving Lemma \ref {lem:link}.

\begin {proof}[Proof of Lemma \ref {lem:link}]
Let $ \sigma^2$ be the variance of $ \mu$. Note that $ \sigma^2 = \infty$ if $ \theta \in (1,2)$, $ \sigma^2 < \infty$ if $ \theta >2$ and that we can have either $ \sigma^2 = \infty$ or $ \sigma^2 < \infty$ for $ \theta=2$. When $ \sigma^2= \infty$, the desired result follows from classical results expressing $B_n$ in terms of $\mu$. Indeed, in the case $ \theta<2$, we may choose $B_n$ and $B'_n$ such that (see e.g. \cite[Theorem 1.10]{Kor12}):
$$ \frac{B'_n}{B_n } = \frac{ \inf \left\{ x \geq 0; \,  \Prmu { \zt \geq x} \leq
\frac{1}{n}\right\} }{\inf \left\{ x \geq 0; \,  \mu([x, \infty)) \leq
\frac{1}{n}\right\}}.$$
Property (ii) in Assumption $(H_ \theta)$ and \eqref{eq:zn} entail that $ { \Prmu { \zt \geq x}}/{\mu([x, \infty))} \rightarrow  1/\gamma ^ {1+ \theta}$ as $x \rightarrow \infty$. The result easily follows. The case when $ \sigma^2= \infty$ and $ \theta=2$ is treated by using similar arguments. We leave details to the reader.

We now concentrate on the case $ \sigma^2< \infty$. Note that necessarily $ \theta \geq 2$. Let $ \sigma'^2$ be the variance of $\zt$ under $ \Pmu$ (from \eqref{eq:zn} this variance is finite when $ \sigma^2 < \infty$). We shall show that $ \sigma'= \sigma/ \gamma^ {3/2}$. The desired result will then follow  since we may take $B_n= \sigma \sqrt {n/2}$ and $ B'_n= \sigma' \sqrt {n/2}$ by the classical central limit theorem. In order to calculate $ \sigma'^2$, we introduce the Galton--Watson process $( \mathfrak{Z}_i)_ {i \geq 0}$ with offspring distribution $ \mu$ such that $ \mathfrak {Z}_0=1$. Recall that $ \Es{\mathfrak{Z}_i}= \mathfrak {m}^i$. Then note that:$$\sigma'^2= \Esmu { \zt^2}- \Esmu { \zt}^2= \Es { \left( \sum_ {i=0}^ \infty \mathfrak {Z}_i \right)^2}- \frac{1}{ \gamma^2}.$$
Since $( \mathfrak{Z}_i/ \mathfrak{m}^	i)_ {i \geq 0}$ is a martingale with respect to the filtration generated by $( \mathfrak{Z}_i)_ {i \geq 0}$, we have $ \Es {Z_i Z_j}= \mathfrak {m}^ {j-i} \Es {Z_i^2}$. Also using the well-known fact that for $i \geq 1$ the variance of $\mathfrak{Z}_i$ is $ \sigma^2 \mathfrak {m}^ {i-1}(\mathfrak {m}^i-1)/(\mathfrak {m}-1)$ (see e.g. \cite[Section 1.2]{AN72}), write:
\begin{eqnarray*}
\Es { \left( \sum_ {i=0}^ \infty \mathfrak {Z}_i \right)^2} &=& \sum_  {i=0}^ \infty \Es {\mathfrak {Z}_i ^2}+ 2 \sum_ {0 \leq i<j}  \mathfrak {m}^ {j-i}\Es {\mathfrak {Z}_i ^2}=\sum_  {i=0}^ \infty \Es {\mathfrak {Z}_i ^2} \left( 1+ \frac{2 \mathfrak {m}}{1-\mathfrak {m}} \right) \\
&=& \left(1+\sum_  {i=1}^ \infty  \left(\frac{\sigma^2 \mathfrak {m}^ {i-1}(\mathfrak {m}^i-1)}{\mathfrak {m}-1}+ \mathfrak {m}^ {2i}\right)\right) \left( 1+ \frac{2 \mathfrak {m}}{1-\mathfrak {m}} \right) \\
&=& \frac{ \sigma^2}{ \gamma^3}+\frac{ 1}{ \gamma^2}.
\end{eqnarray*}
This entails $ \sigma'= \sigma/ \gamma^ {3/2}$ and the conclusion follows.
\end {proof}

\begin {proof}[Proof of Corollary \ref{cor:intro}] Recall that  $ \overline { \Delta}(Z)= \sup_{0<s<1}(Z_{s}-Z_{s-})$ denotes the largest jump of $Z \in \D([0,1], \R)$.  It follows from the continuity of $Z \rightarrow \overline{ \Delta}(Z)$ and Theorem \ref {thm:fluctenfants} that:
$$  \frac{ \d1}{ B_n}\max_ {1 \leq i \leq \Delta( \t_n)} \xi_i( \t_n)  \quad\mathop{\longrightarrow}^ {(d)}_{n \rightarrow \infty} \quad  \frac{1}{\gamma }  \sup_ {s \in (0,1]} ( Y_s-Y_{s-}).$$
If $ \theta \geq 2$, $Y$ is continuous and the first assumption of  Corollary \ref{cor:intro} follows. If $ \theta<2$, the result easily follows from the fact that the Lévy measure of $Y$ is $ \nu(dx)= \mathbbm {1}_ {  \{x >0 \}}dx/( \Gamma(- \theta)x^ {1+ \theta})$.
\end {proof}

\subsection {Height of large conditioned non-generic trees}

We now prove Theorem \ref {thmintro:height}. If $ \mathbf {f} = ( \tau_1, \ldots, \tau_k)$ is a forest, its height $\H( \mathbf { \bf})$ is by definition $ \max( \H( \tau_1), \ldots, \H( \tau_k))$. Recall that for $1 \leq i \leq \Delta( \tau)$, let $\mathcal{T}_{i}( \tau)$ is the tree of descendants of the $i$-th child of $ u_{ \star}( \tau)$ and that $\mathcal{F}_{i,j}(\tau)=(\mathcal{T}_{i}( \tau), \ldots, \mathcal{T}_{j}( \tau))$.

\begin {proof}[Proof of Theorem \ref{thmintro:height}]If $ \tau$ is a tree, let ${ \H}_ \star ( \tau)= 1+\H \left( \mathcal{F}_{1, \Delta(\tau)}(\tau) \right)$ be the height of the subtree of descendants of $u_ \star (\tau)$ in $ \tau$. By Theorem \ref {thm:cvd} (ii), the generation of ${u}_ \star ( \t_n)$ converges in distribution. It is thus sufficient to establish that, if  $( \lambda_n)_ { n \geq 1}$ of positive real numbers tending to infinity:
\begin{equation}
\label{eq:montrerheight}\Pr { \left|{ \H}_ \star( \t_n) -\frac{ \ln (n)}{ \ln(1/\mathfrak {m})} \right| \leq  \lambda_n}  \quad\mathop{\longrightarrow}_{n \rightarrow \infty} \quad1.
\end{equation}
To simplify notation, set
$ \H^{(n)}_ {i,j}=   \H(\mathcal {F}_ {i,j}( \t_{n}))$   and  $p_n=  \ln(n)/ \ln(1/\mathfrak {m})- \lambda_n$. Let us first prove the lower bound, that is $\Pr{{ \H}_ \star( \t_n) \leq p_n}  \rightarrow 0$ as $n \rightarrow \infty$. It is plain that $\Pr{{ \H}_ \star( \t_n) \leq p_n} \leq  \Pr{ \H^{(n)}_{1, \fl{ \gamma n/2}} \leq p_{n}}$. In addition,  by Corollary \ref{cor:indep}, 
$$ \Pr{ \H^{(n)}_{1, \fl{ \gamma n/2}} \leq p_{n}} - \mathbb{P}_{\mu,\fl{ \gamma n/2}}\left( \H(\mathbf{f}) \leq p_{n} \right)  \quad\mathop{\longrightarrow}_{n \rightarrow \infty} \quad 0.$$
But $$\mathbb{P}_{\mu,\fl{ \gamma n/2}}\left( \H(\mathbf{f}) \leq p_{n} \right) = (1-\Prmu { \H ( \tau) > p_n})^ {\fl{\gamma n/2}}.$$
Since $\mu$ satisfies Assumption $(H_ \theta)$, we have $ \sum_ {i \geq 1} {i \ln (i) \mu_i} < \infty$ . It follows from \cite[Theorem 2]{HSV67} that there exists a constant $c>0$ such that:
\begin{equation}
\label{eq:esH} \Prmu { \H( \tau) > k}  \quad\mathop{ \sim}_{k \rightarrow \infty} \quad  c \cdot \mathfrak {m}^ {k}.
\end{equation} 
Hence $\Prmu { \H ( \tau) > p_n}) \sim  c \cdot {1}/( n \cdot \mathfrak{m}^ { \lambda_{n}})$ as $n \rightarrow \infty$. Consequently $\mathbb{P}_{\mu,\fl{ \gamma n/2}}\left( \H(\mathfrak{f}) \leq p_{n} \right)$ tends to $0$ as $n \rightarrow \infty$,  and the proof of the lower bound is complete.

Now set $q_n=  \ln(n)/ \ln(1/\mathfrak {m})+ \lambda_n$. The proof of the fact that $\Pr{{ \H}_ \star ( \t_n) \geq q_n}  \rightarrow 0$ as $ n \rightarrow \infty$ is similar and we only sketch the argument. Write:
$$\Pr{{ \H}_ \star ( \t_n) \geq q_n} \leq  \Prmu{ \H^{(n)}_ {1, \fl{ \Delta(\t_{n}) /2}} \geq q_n}+  \Prmu { \H^{(n)}_ { \fl{\Delta(\t_{n}) /2}+1, \Delta(\t_{n})} \geq q_n}$$
Since  $\H^{(n)}_ { \Delta(\t_{n})+1,\Delta(\t_{n})}$ has the same distribution as $\H^{(n)}_ { 1, \Delta(\t_{n})- \fl {\Delta(\t_{n})/2}}$, it suffices to show that the first term of the last sum tends to $0$ as $n \rightarrow \infty$. By Theorem \ref{thmintro:1} (i), we have $ \Delta( \t_{n})/2 \leq  \fl{2 \gamma n/3}$ with probability tending to $1$ as $n \rightarrow \infty$. It is thus sufficient to establish that  $\Prmu{ \H^{(n)}_ {1, \fl{2 \gamma n /3}} \geq q_n} \rightarrow 0$ as $n \rightarrow \infty$. By arguments similar to those of the proof of the lower bound, it is enough to check that
$$\mathbb {P}_ { \mu, \fl {2\gamma n/3}}  \left( \H( \bf) \geq q_n\right)  \quad\mathop{\longrightarrow}_{n \rightarrow \infty} \quad 0.$$
This follows from \eqref{eq:esH}, the fact that $\mathbb {P}_ { \mu, \fl {2\gamma n/3}}  \left( \H( \bf) \geq q_n\right)  =1-(1-\Prmu { \H( \tau) \geq  q_n})^ { \fl{ 2 \gamma n/3}}$ combined with the asymptotic behavior $\Prmu { \H( \tau) \geq  q_n }\sim c \cdot \mathfrak{m}^{ \lambda_{n}}/n$ as $ n \rightarrow \infty$. This completes the proof of the upper bound and establishes \eqref{eq:montrerheight}.\end{proof}

Theorem \ref {thmintro:height} implies that $ \H( \t_n)/  \ln(n) \to \ln(1/\mathfrak {m})$ in probability as $n \to \infty$. We next show that this convergences holds in $ \L^p$ for every $p \geq 1$.

\begin{prop}\label{prop:lp}For every $p \geq 1$, we have
$$\Es {\H( \t_n)^p} \quad\mathop{ \sim}_{n \rightarrow \infty} \quad \frac{ \ln(n)^p}{ \ln(1/\mathfrak {m})^p}.$$
\end{prop}

\begin{proof}The following proof is due to an anonymous referee. It is sufficient to show there exists $K>0$ such that
$$ \Es {\H( \t_n)^p \mathbbm{1}_{  \{\H( \t_n) > K \ln(n)\}}}  \quad\mathop{\longrightarrow}_{n \rightarrow \infty} \quad 0 .$$ 
By \eqref{eq:zn}, we have $ \Pr{ \zt=n} \geq n^{-2- \theta}$ for $n$ sufficiently large, and in addition by \eqref{eq:esH} we  have $\Pr {\H( \tau) > K \ln(n)} \leq 2c \cdot n^{- K \ln(1/\mathfrak{m})}$ for $n$ sufficiently large. Hence, bounding the height of $ \H( \t_{n})$ by $n$, we get
\begin{eqnarray*}
 \Es {\H( \t_n)^p \mathbbm{1}_{  \{\H( \t_n) > K \ln(n)\}}}   & \leq & n^p \cdot \Pr{\H( \t_n) > K \ln(n)} \\
 &\leq&  n^p  \cdot { \Pr {\H( \tau) > K \ln(n)}} /{ \Pr{ \zt=n}} \\
 & \leq &  2c \cdot n^{p+2+ \theta-K \ln(1/ \mathfrak{m})}.
\end{eqnarray*}
It thus suffices to chose $K>(p+2+ \theta)/ \ln(1/ \mathfrak{m})$. This completes the proof.
\end{proof}

\subsection{Scaling limits of non-generic trees}

We turn to the proof of Theorem \ref {thm:GH}. 

\begin {proof}[Proof of Theorem \ref {thm:GH}.] Fix $ \eta \in (0,1/ \ln(1/ \mathfrak{m}))$. We shall show that, with probability tending to one as $n \rightarrow \infty$, at least $ \ln(n)$ trees among the $ \fl {  \gamma n /2}$ trees   $\mathcal{T}_{1}( \t_{n}), \ldots \mathcal{T}_{\fl{ \gamma n /2}}( \t_{n})$ have height at least $ \eta \ln(n)$. This will indeed show that, with probability tending to one as $n \rightarrow \infty$, the number of balls of radius less than $ \eta$ needed to cover $ \ln(n)^{-1} \cdot \t_{n}$ tends to infinity. By standard properties of the Gromov--Hausdorff topology (see \cite[Proposition 7.4.12]{BBI01}) this implies that the sequence of random metric spaces $(\ln(n)^{-1} \cdot \t_{n})_{n \geq 1}$ is not tight.

If $ \textbf{f}=(\tau_{1}, \ldots, \tau_{j})$ is a forest, let $ \mathsf{E}_n( \textbf{f})$ be the event defined by
$$  \mathsf{E}_n( \textbf{f})=  \{ \textrm{at most } \ln(n) \textrm{ trees among } \tau_{1}, \ldots, \tau_{j} \textrm{ have height at least } \eta \ln(n)\}.$$
It is thus sufficient to prove that $\Pr{ \mathsf{E}_n(\mathcal{T}_{1}( \t_{n}), \ldots \mathcal{T}_{\fl{ \gamma n /2}}( \t_{n}))}$ converges toward $0$ as $n \rightarrow \infty$. As previously, by Corollary \ref{cor:indep}, it is sufficient to establish that
\begin{equation}
\label{eq:amqlim}\P_{ \mu, \fl{ \gamma n/2}} \left( \mathsf{E}_n( \textbf{f}) \right)  \quad\mathop{\longrightarrow}_{n \rightarrow \infty} \quad 0.
\end{equation}
Now denote by $N_{n}$ the number of trees among a forest of $ \fl{ \gamma n/2}$ independent $ \GW_{ \mu}$ trees of height at least $ \eta \ln(n)$. Using \eqref{eq:esH} and setting $ \eta '= \eta \ln (1/m)$, we get that for a certain constant $C>0$, $N_{n}$ dominates a binomial random variable $ \textsf{Bin}( \fl{ \gamma n/2}, C n^{ \eta'})$, which easily implies that $ \Pr{N_{n} \leq  \ln(n)} \rightarrow 0$ as $ n \rightarrow \infty$. This shows \eqref{eq:amqlim} and completes the proof.
\end {proof}

Note that  Theorem \ref {thm:GH} implies that there is no nontrivial scaling limit for the contour function coding $ \t_{n}$, since convergence of scaled contour functions imply convergence in the Gromov--Hausdorff topology (see e.g. \cite[Lemma 2.3]{LG05}).

\subsection{Finite dimensional marginals of the height function}

We first extend the definition of the height function  to a forest. If $ \textbf{f}= ( \tau_{i})_{1 \leq i \leq  j}$ is a forest, set $n_0=0$ and $n_{p}= | \tau_{1}|+|\tau_{2}|+ \cdots +|\tau_{p}|$ for $1 \leq p  \leq j$. Then, for every $0 \leq i \leq p-1$ and $0 \leq k \leq |\tau_{i+1}|$, set
$$ \textbf{H}_{n_{i}+ k}( \textbf{f})=H_{k}( \tau_{i+1}).$$
Note that the excursions of $ \textbf{H}( \textbf{f})$ above $0$ are the $( \textbf{H}_{n_{i}+k}( \textbf{f}); 0 \leq k \leq  |\tau_{i+1}| )$. The Lukasiewicz path $ \W( \textbf{f})$ and height function $ \textbf{H}( \textbf{f})$ satisfy the following relation  (see e.g. \cite[Proposition 1.7]{LG05} for a proof): For every $0 \leq n \leq | \textbf{f}|$,

\begin{equation}
\label{eq:relHW}\mathbf {H}_{n}( \textbf{f})= \textrm{Card}(  \{k \in  \{0,1, \ldots,n-1\} ; \, \W_{k}( \textbf{f})= \inf_{k \leq j \leq n} \W_{j}( \textbf{f})\}).
\end{equation}



Recall that $(W_{n})_{n \geq 0}$ stands for  the random walk introduced in Proposition \ref {prop:RW} with $ \rho= \mu$ and that $ X_k = W_k - W_ {k-1}$ for $k \geq 1$. For every $n \geq 0$, set
$$ {H}_{n}= \textrm{Card}(  \{0 \leq k \leq n-1 ; \, W_{k}= \inf_{k \leq j \leq n} W_{j} \}), \quad J_{n}= n-\min  \{ 0 \leq i \leq n; \quad { W}_{i} =  \min_{0 \leq j \leq n} W_{j} \}. $$
Finally, for $k \geq 0$, set $$M_{k}= \textrm{Card}(  \{ 1 \leq  i \leq  k ; \, W_{i} = \max_{0 \leq j \leq i } W_{j} \}), \qquad T= \sup \{i \geq 0, W_{i}= \sup_{j \geq 0} W_{j} \}.$$ Since $W$ drifts almost surely to $- \infty$, $T$ is almost surely finite, and $M_{T}$ is distributed according to a geometric random variable of parameter $ \Pr { \forall i \geq 1, W_{i} \leq -1}= \gamma$, by Proposition \ref{prop:useful} (i).

The following result, which is an unconditioned version of Theorem \ref{thm:cvfinidim}, will be useful.  

\begin {lem} \label{lem:fd}For every $0 < s  <1$, the following convergence holds in distribution:
\begin{equation}
\label{eq:cvvers} \left(H_{ \fl{ns}},H_{n},J_n\right)  \quad\mathop{\longrightarrow}^ {(d)}_{n \rightarrow \infty} \quad ( \mathbf{e_{1}}, M_T,T),
\end{equation}
where  $\mathbf{e_{1}}$ is geometric random variable of parameter $ \gamma$, independent of $(X_{n})_{n \geq 1}$.
\end {lem}

\begin{proof}
 Set  $W^{n}_{i}=W_{n-\fl{ns}+i}-W_{n-\fl{ns}}$ for $i \geq 0$, and $M^{n}_{k}= \textrm{Card}(  \{ 1 \leq  i \leq  k ; \, W^{n}_{i} = \max_{0 \leq j \leq i } W^{n}_{j} \})$. Set also $$T_{n}= \max \{i \in \{0,1, \ldots,n\} , W_{i}= \sup_{ 0 \leq j \leq n} W_{j} \}.$$ Notice that $(W^{n}_{i}, i \geq 0)$ has the same distribution as $(W_{i}, i \geq 0)$. Using the fact that $(W_{i}, 0 \leq i \leq n)$ and $ (W_{n}-W_{n-i}, 0 \leq i \leq  n)$ have the same distribution, we get that
 $$ \left(H_{ \fl{ns}},H_{n},J_n\right) \quad\mathop{=}^{(d)} \quad   \left( M^{n}_{ \fl{ns}}, M_{n},T_n \right).$$
 Let $F_{1}: \Z \rightarrow \R_{+}$, $F_{2}: \Z^2 \rightarrow \R_{+}$ be bounded functions and fix $\epsilon>0$. 
Choose $N_{0}>0$ such that  $ \Pr {T>N_{0}}< \epsilon$.  For $n \geq N_{0}$ , note that  $M_{n}=M_{ N_{0}}$  and $T_n=T_{N_0}$ on the event $T \leq N_{0}$. Hence for $n \geq N_0$:
 
 $$ \left|\Es{F_{1}(M^{n}_{ \fl {ns}}) F_{2}(M_{n},T_n)} - \Es { F_{1}(M^{n}_{ \fl {ns}}) F_{2}(M_{N_{0}},T_{N_{0}}) }\right| \leq C \epsilon$$
where $C>0$ is a constant depending only on $F_{1},F_{2}$ (and which may change from line to line). Next, using the fact that $M^{n}_{ \fl{ns}}$ is independent of $(M_{N_0},T_{N_0})$ for $n> N_0/(1-s)$, we get that for  $n> N_0/(1-s)$,
  $$ \left|\Es{F_{1}(M^{n}_{ \fl {ns}}) F_{2}(M_{n},T_n)} - \Es { F_{1}(M^{n}_{ \fl {ns}})} \Es{ F_{2}(M_{N_{0}},T_{N_{0}}) }\right| \leq C \epsilon.$$
 The conclusion immediately follows since $M^{n}_{ \fl {ns}}$ has the same distribution as $M_{ \fl {ns}}$ and since $(M_n,T_n)$ converges in distribution toward $(M_T,T)$ as $n \to \infty$.
\end{proof}

\begin {rem}\label{rem:gen2}It is straightforward to adapt the proof of Lemma \ref{lem:fd} to get that for every $0 < t_{1} < t_{2} < \cdots <t_{k}<1$ and $b>0$,  the following convergence holds in distribution:
$$ \left(H_{ \fl{n t_{1}-b}},H_{\fl{n t_{2}-b}}, \ldots, H_{\fl{n t_{k}-b}}, H_{n-1},J_{n-1} \right)  \quad\mathop{\longrightarrow}^ {(d)}_{n \rightarrow \infty} \quad ( \mathbf{e_{1}}, \mathbf{e_{2}}, \ldots, \mathbf{e_{k}},M_T,T)$$
where $ (\mathbf{e_{i}})_{1 \leq i \leq k}$ are i.i.d.~geometric random variables of parameter $ \gamma$, independent of $(M_T,T)$.
\end{rem}

Recall that for $1 \leq i \leq j \leq  \Delta( \tau)$,  $\mathcal{T}_{i}( \tau)$ is the tree of descendants of the $i$-th child of $ u_{ \star}( \tau)$, that $\mathcal{F}_{i,j}(\tau)=(\mathcal{T}_{i}( \tau), \ldots, \mathcal{T}_{j}( \tau))$ and that  $ \widetilde{\z}_{k}(\tau)= \inf  \{i \geq 0; \widetilde{ \mathcal{W}}_{i}(\tau)=-k\}$ for $1 \leq k \leq  \Delta(\tau)$. We are now ready to prove Theorem \ref{thm:cvfinidim}. 

\begin {proof}[Proof of  Theorem \ref{thm:cvfinidim}]

To simplify, we establish Theorem \ref{thm:cvfinidim} for $k=2$, the general case being similar. To this end, we fix $0<s<t<1$ and shall show that
\begin{equation}
\label{eq:mqdim}(H_{ \fl{n s}}(\t_n), H_{\fl{n t}}(\t_n))\quad\mathop{\longrightarrow}^{(d)}_{n \rightarrow \infty} \quad ( 1+\mathbf{e_{0}}+\mathbf{e_{1}}, 1+\mathbf{e_{0}}+\mathbf{e_{2}}).
\end{equation}
We first express $H_{ \fl{n s}}(\t_n)$ in terms of the modified Lukasiewicz path $\widetilde{\W}$ which was defined in Section \ref{sec:descr}. To this end we need to introduce some notation. For every tree $ \tau$ and $0 \leq p \leq | \tau|-1$, set $$ \widetilde{H}_{p}( \tau)= \textrm{Card}(  \{k \in  \{0,1, \ldots,p-1\} ; \, \widetilde{W}_{k}( \tau)= \inf_{k \leq j \leq p} \widetilde{W}_{j}( \tau)\}).$$
Note that  by Proposition \ref {prop:links} (ii), $(\widetilde{H}_{1}( \t_n), \ldots, \widetilde{H}_{\widetilde{\z}_{\Delta(\t_n)}( \t_n)}( \t_n))$ is the height function of the forest $\mathcal{F}_{ 1, \Delta( \t_n)}( \t_n)$. For every $n \geq 1$ and  $r \in (0,1)$ such that $U(  \t_{n})<\fl{nr}<\widetilde{\z}_{\Delta(\t_{n})}( \t_{n})$, we have $H_{ \fl{n r}}(\t_n)=1+H_{ \fl{n r}-U(  \t_{n})-1}(\mathcal{F}_{ 1, \Delta( \t_{n})}( \t_{n}))+|u_{ \star}( \tau)|$. Hence, using Proposition \ref {prop:links} (ii) and \eqref{eq:relHW}:
$$H_{ \fl{n r}}(\t_n)=1+\widetilde{H}_{\fl{n r}-U( \t_{n})-1}( \t_{n})+ \widetilde{H}_{n-1}( \t_n).$$

Since $U( \t_n)$ and $\widetilde{\z}_{\Delta(\t_{n})}( \t_{n})$ converge in distribution (by respectively Theorem \ref{thm:cvd} and Remark \ref{rem:cl}), we have  $U(  \t_{n})<\fl{ns}<\fl{nt}<\widetilde{\z}_{\Delta(\t_{n})}( \t_{n})$ with probability tending to $1$ as $n \to \infty$. By combining Proposition \ref{prop:links} (i) and Proposition \ref{prop:tilde}, we get that:
$$ \sup_{A \in \mathcal{B}( \R^2)} \left| \Pr{(H_{ \fl{n s}}(\t_n), H_{\fl{n t}}(\t_n)) \in A}- \Pr{(1+H_{ \fl{ns-J_{n-1}-1}} + H_{n-1},1+H_{ \fl{nt-J_{n-1}-1}} + H_{n-1}) \in A } \right|$$
converges to $0$ as $n \rightarrow \infty$.
But by Remark \ref{rem:gen2}, we have $$(1+H_{ \fl{ns-J_{n-1}-1}} + H_{n-1},1+H_{ \fl{nt-J_{n-1}-1}} + H_{n-1})   \quad\mathop{\longrightarrow}^ {(d)}_{n \rightarrow \infty} \quad ( 1+\mathbf{e_{1}}+M_T, 1+\mathbf{e_{2}}+M_T) $$
with $M_T$ independent of $\mathbf{e_{1}},\mathbf{e_{2}}$. Since $M_{T}$ is distributed according to a geometric random variable of parameter $ \gamma=1-\mathfrak {m}$, the conclusion immediately follows.
\end {proof}

\section{Extensions and comments}

We conclude by proposing possible extensions and stating a few open questions.

\medskip

\textbf {Other types of conditioning.} Throughout this text, we have only considered the case of Galton--Watson trees conditioned on having a fixed total progeny. It is natural to consider different types of conditioning.  For instance, for $ n \geq 1$, let $ \t ^ h_n$ be a random tree distributed according to $ \Prmu { \, \cdot \, | \H( \tau ) \geq n}$. In \cite[Section 22]{Jan12}, Janson has in particular proved that when $ \mu$ is critical or subcritical, as $n \rightarrow \infty$, $ \t^h_n$ converges locally to Kesten's Galton--Watson tree conditioned to survice $ \mathcal {T}^*$, which a random infinite tree  different from $ \widehat{\mathcal {T}}$. It would be interesting to know whether the theorems of the present work apply in this case.

Another type of conditioning involving the number of leaves has been introduced in \cite {CKdissections,Kor12,Riz11}. If $ \tau$ is a tree, denote by $ \lambda( \tau)$ the number of leaves of $ \tau$ (that is the number of individuals with no child). For $ n \geq 1$ such that $ \Prmu { \lt=n}>0$, let $ \t^l_n$ be a random tree distributed according to $ \Prmu { \, \cdot \, | \l( \tau) = n}$. Do results similar to those we have obtained hold when $ \t_n$ is replaced by $\t^l_n$? We expect the answer to be positive, since a $ \GW_ \mu$ tree with $n$ leaves is very close to a $ \GW_ \mu$ with total progeny $n/ \mu_0$ (see \cite {Kor12} for details), and we believe that the techniques of the present work can be adapted to solve this problem.

\medskip

\textbf {Concentration of $\H( \t_n)$ around ${ \ln (n)}/{ \ln(1/\mathfrak {m})} $.} By Theorem \ref {thmintro:height}, the sequence of random variable $ (\H( \t_n) -{ \ln (n)}/{ \ln(1/\mathfrak {m})})_ {n \geq 1} $ is tight. It is therefore natural to ask the following question, due to Nicolas Broutin. Does there exist a random variable $ \mathscr{H}$ such that:
$$\H( \t_n) - \frac{ \ln (n)}{ \ln(1/\mathfrak {m})}  \quad\mathop{\longrightarrow}^ {(d)}_{n \rightarrow \infty} \quad \mathscr {H} \quad?$$
We expect the answer to be negative. Let us give a heuristic argument to support this prediction. In the proof of Theorem \ref {thmintro:height}, we have seen that the height of  $\H( \t_n)$ is close to the height of $ \fl {\gamma n}$ independent $ \GW_ \mu$ trees and the height of each of these trees satisfies the estimate \eqref{eq:esH}. However, if $(Q_i)_ {i \geq 1}$ is an i.i.d. sequence of random variables such that $ \Pr {Q_1 \geq k}= c  \cdot \mathfrak {m}^k$, then it is known (see e.g. \cite[Example 4.3]{Jan06}) that the random variables $$\max(Q_1,Q_2, \ldots,Q_n)- \frac{ \ln (n)}{ \ln(1/\mathfrak {m})}$$
do not converge in distribution.

\medskip

\textbf {Other types of trees}. Janson \cite {Jan12} gives a very general limit theorem concerning the local asymptotic behavior of simply generated trees conditioned on having a fixed large number of vertices. Let us briefly recall the definition of simply generated trees. Fix a sequence $ \textbf {w}= ( w_k)_ {k \geq 0}$ of nonnegative real numbers such that $ w_0>0$ and such that there exists $k>1$ with $w_k>0$ ($\textbf {w}$ is called a weight sequence). Let $ \mathbb {T}_f \subset \T$ be the set of all finite plane trees and, for every $n \geq 1$, let $ \mathbb {T}_n$ be the set of all plane trees with $n$ vertices. For every $ \tau \in \mathbb {T}_f $, define the weight $ w( \tau)$ of $ \tau$ by:
$$w ( \tau)= \d \prod_ {u \in \tau} w_ {k_u( \tau)}.$$
Then for $n \geq 1$ set
$$Z_n = \sum_ { \tau \in \mathbb {T}_n}  w ( \tau).$$
For every $ n \geq 1$ such that $Z_n \neq 0$, let $ \mathcal {T}_n$ be a random tree taking values in $ \T_n$ such that for every $ \tau \in \T_n$:
$$ \Pr {\mathcal {T}_n= \tau} = \frac{ w ( \tau)}{Z_n}.$$
The random tree $ \mathcal{T}_n$ is said to be finitely generated. Galton--Watson trees conditioned on their total progeny are particular instances of simply generated trees. Conversely, if $ \mathcal {T}_n$ is as above, there exists an offspring distribution $ \mu$ such that $\mathcal {T}_n$ has the same distribution as a $ \GW_ \mu$ tree conditioned on having $n$ vertices if, and only if, the radius of convergence of $ \sum w_i z^i$ is positive (see \cite[Section 8]{Jan12}).

It would thus be interesting to find out if the theorems obtained in the present work for Galton--Watson trees can be extended to the setting of simply generated trees whose associated radius of convergence is $0$. In the latter case, Janson \cite {Jan12} proved that $ \mathcal {T}_n$ converges locally as $n \rightarrow \infty$ toward a deterministic tree consisting of a root vertex with an infinite number of leaves attached to it. We thus expect that the asymptotic properties derived in the present work will take a different form in this case. We hope to investigate this in future work.


 \begin {tabular}{l }
Laboratoire de mathématiques,  UMR 8628 CNRS, Université Paris-Sud \\
91405 ORSAY Cedex, France
\end {tabular}

\medbreak
\noindent \texttt{igor.kortchemski@normalesup.org}
\end{document}